\documentclass[11pt,a4paper]{article}

\usepackage{inputenc}
\usepackage{amsmath}
\usepackage{bm}
\usepackage{bbold}
\usepackage{amsthm}
\usepackage{enumerate}

\usepackage[hyphens]{url}

\usepackage{hyperref}
\usepackage{breakurl}

\usepackage{tikz}\tikzset{x=1cm,y=1cm,z=1cm}
\usetikzlibrary{perspective}

\usepackage{pgfplots}\pgfplotsset{compat=1.16}

\title{Application of tropical optimization for solving multicriteria problems of pairwise comparisons using log-Chebyshev approximation\thanks{
Internat. J. Approx. Reason. 169, 109168 (2024)}
}

\author{N. Krivulin\thanks{Faculty of Mathematics and Mechanics, Saint Petersburg State University, 28 Universitetsky Ave., St.~Petersburg, 198504, Russia, nkk@math.spbu.ru.}
}

\date{}

\newtheorem{theorem}{Theorem}
\newtheorem{lemma}[theorem]{Lemma}

\theoremstyle{definition}

\begin{document}

\maketitle

\begin{abstract}
We consider a decision-making problem to find absolute ratings of alternatives that are compared in pairs under multiple criteria, subject to constraints in the form of two-sided bounds on ratios between the ratings. Given matrices of pairwise comparisons made according to the criteria, the problem is formulated as the log-Chebyshev approximation of these matrices by a common consistent matrix (a symmetrically reciprocal matrix of unit rank) to minimize the approximation errors for all matrices simultaneously. We rearrange the approximation problem as a constrained multiobjective optimization problem of finding a vector that determines the approximating consistent matrix. The problem is then represented in the framework of tropical algebra, which deals with the theory and applications of idempotent semirings and provides a formal basis for fuzzy and interval arithmetic. We apply methods and results of tropical optimization to develop a new approach for handling the multiobjective optimization problem according to various principles of optimality. New complete solutions in the sense of the max-ordering, lexicographic ordering and lexicographic max-ordering optimality are obtained, which are given in a compact vector form ready for formal analysis and efficient computation. We present numerical examples of solving multicriteria problems of rating four alternatives from pairwise comparisons to illustrate the technique and compare it with others.
\\

\\

\end{abstract}

\section{Introduction}

Decision-making problems that are encountered in real-world applications often have to cope uncertain data and involve multiple objectives. To handle uncertainty in data in these problems, various approaches are adopted, including the applications of interval, fuzzy and possibilistic data models \cite{Bellman1970Decisionmaking,Dubois1980Fuzzy,Carlsson2011Possibility}. A key component of many approaches that aim to address uncertainty is the use of the arithmetics of idempotent semirings and semifields, which provides a bridge between these approaches and tropical mathematics. The interaction between tropical (idempotent) mathematics which concentrates on the theory and applications of algebraic systems with idempotent operations, and the fuzzy, interval and possibilistic mathematics appears to be bidirectional. On the one hand, fuzzy and possibilistic models have inspired the study of idempotent algebraic systems such as fuzzy algebra, {\L}ukasiewicz semirings and Viterbi semirings (see, e.g. \cite{Golan2003Semirings,Gondran2008Graphs,Dinola2015Semiring,Gavalec2015Decision}). On the other hand, the tropical mathematics offers a conceptual and analytical framework for computations in fuzzy, interval and other arithmetics that involve $\max$ and $\min$ operations. Examples of related recent studies include \cite{Gavalec2015Tropical,Valverdealbacete2015Spectra,Krivulin2019Tropical,Shitov2020Factoring}. The above observations indicate that the development and investigation of novel solutions in tropical mathematics can benefit the fuzzy sets theory and other research domains where idempotency plays an essential role to deal with uncertainty.

As another bridge between the tropical mathematics and the mathematics of uncertainty, we consider the concept of Chebyshev distance and related theory of Chebyshev approximation. In tropical semifields, Chebyshev distance is introduced in a direct way using only internal operations, which offers a strong potential for developing efficient techniques to solve optimization problems involving Chebyshev approximation. At the same time, the methods and techniques of Chebyshev approximation appear to be in growing demand as useful approach to solve optimization and other problems in fuzzy mathematics \cite{Dubois1999Computing,Liu2008Multiobjective,Li2010Chebyshev}, which makes the related results in tropical mathematics relevant.  

The known methods to address uncertain decision problems under multiple criteria include both crisp and fuzzy implementations of regular multiobjective optimization methods and heuristic algorithms \cite{Luhandjula1989Fuzzy,Buckley1990Multiobjective,Fiedler2006Linear,Elwahed2008Intelligent}. One of the approaches is to use methods and techniques of tropical optimization which is concerned with optimization problems that are formulated and solved in terms of tropical mathematics. Due to the clear connection between the fuzzy and tropical mathematics, the development of new solutions that are based on multiobjective tropical optimization suggests a promising resource to handle uncertain multicriteria decision problems, for example through a fuzzy implementation of the optimization method.

The problem of evaluating ratings (scores, priorities, weights) of alternatives based on pairwise comparisons under multiple criteria is an uncertain decision problem, where the uncertainty arises from subjective judgment errors and imprecise data of comparisons. This uncertainty leads to inconsistent or opposite results of rating and ranking alternatives, and needs to be managed to achieve acceptable solutions. The problem under consideration is widely occurring and highly demanded in decision making, and has been studied for decades in many researches including the classical paper by L.~L.~Thurstone \cite{Thurstone1927Law}. Given results of comparisons of decision alternatives in the form of pairwise comparison matrices, the problem is to find a vector of individual ratings of alternatives, which can be used to rank alternatives according to their ratings and thus provide a basis for optimal decisions. 

The solution techniques available for the problem apply various heuristic procedures that do not guarantee the optimal solution, but usually produce acceptable results in practice, and mathematical methods that are formally justified, but may be computationally very expensive. Along with the methods based on conventional algebra, there exist solutions that use interval arithmetic, fuzzy arithmetic and tropical algebra.

One of the most prominent approach to solve multicriteria pairwise comparison problems, which is known as the analytical hierarchy process, has been proposed in the 1970s by T.~L.~Saaty \cite{Saaty1977Scaling,Saaty1990Analytic,Saaty2013Onthemeasurement}. The approach employs the principal eigenvector method to derive a vector of individual ratings by calculating the eigenvectors of the pairwise comparison matrices, corresponding to their maximal eigenvalues (the Perron eigenvectors). In succeeding years the approach was extended to handle problems with interval and fuzzy pairwise comparison matrices \cite{Laarhoven1983Fuzzy,Buckley1985Fuzzy,Saaty1987Uncertainty,Salo1995Preference} (see also literature overviews in \cite{Krejci2018Pairwise,Ramik2020Pairwise}) and to implement tropical algebra \cite{Elsner2004Maxalgebra,Elsner2010Maxalgebra,Gursoy2013Analytic,Tran2013Pairwise,Gavalec2015Decision}.

Another widely used approach is the weighted geometric mean method, which is based on matrix approximation in the sense of the Euclidean metric in logarithmic scale \cite{Narasimhan1982Geometric,Crawford1985Note,Barzilai1987Consistent}. Under this approach, the individual ratings are obtained from a matrix of pairwise comparisons as the vector of geometric means of elements in the rows of the matrix. For interval and fuzzy extensions of the geometric mean method, one can consult \cite{Krejci2018Pairwise,Ramik2020Pairwise}.

Different solution approaches to the pairwise comparison problems may yield results where the acquired ratings produce different or contradictory orders of alternatives \cite{Belton1983Shortcoming,Barzilai1987Consistent,Ishizaka2006How,Mazurek2021Numerical}. Some of the approaches offer analytical solutions with low computational complexity, whereas the other are based on numerical algorithms that may have a sufficient computational cost. As a result, the development of new effective solutions of the problem to supplement and complement existing approaches remains relevant.

A technique that applies approximation of pairwise comparison matrices in Chebyshev metric in logarithmic scale (log-Chebyshev approximation) is proposed and developed in \cite{Krivulin2015Rating,Krivulin2016Using,Krivulin2019Tropical}. The technique leads to optimization problems that are formulated and solved in the framework of tropical algebra where addition is defined as maximum. Using methods and results of tropical optimization yields analytical solutions of the problem in a compact vector form ready for both formal analysis and straightforward computations. In contrast to the methods of principle eigenvector and geometric mean, which offer unique solutions to the problem of pairwise comparisons, the log-Chebyshev approximation approach can produce multiple solutions. At the same time, an important advantage of this approach is that it can be directly extended to solve pairwise comparison problems with constraints imposed on the individual ratings of alternatives, whereas both methods of principal eigenvector and geometric mean cannot accommodate additional constraints without considerable complication of solution.

The log-Chebyshev approximation technique differs from another existing approach based on tropical algebra, which is exploited in many works including \cite{Elsner2004Maxalgebra,Elsner2010Maxalgebra,Gursoy2013Analytic,Tran2013Pairwise,Goto2022Polyad}. This approach can be considered as a tropical analogue of the principal eigenvector method, where the ordinary Perron eigenvector is formally replaced by the tropical eigenvector of the pairwise comparison matrix. The solution obtained as the tropical eigenvector is known to be one of the solutions provided by the log-Chebyshev approximation. However, the tropical eigenvector as shown in \cite{Krivulin2016Using} can yield a less relevant solution in the context of the pairwise comparison problem than other vectors that are found using the approximation. Since the tropical analogue of the principal eigenvector method may miss better results provided by the log-Chebyshev approximation technique, it seems unnecessary to consider this method separately. 

Multicriteria pairwise comparison problems can be considered as optimization problems with multiple objectives and solved using methods and techniques available in multiobjective optimization \cite{Ehrgott2005Multicriteria,Luc2008Pareto,Benson2009Multiobjective,Nakayama2009Sequential,Ramesh2013Multiple}. The most common approach to solving multiobjective optimization problems is to obtain Pareto-optimal (nondominated) solutions that cannot be improved for any one objective without worsening at least one other objective. To obtain Pareto-optimal solutions, various techniques are applied to reduce the problem with a vector objective function to one or more problems with ordinary (scalar) objective functions. As examples, one can consider linear scalarization used in the weighted geometric mean method \cite{Crawford1985Note,Barzilai1987Consistent} and Chebyshev scalarization in the log-Chebyshev approximation method \cite{Krivulin2019Tropical}. Other approaches to solve multiobjective problems include max-ordering, lexicographic ordering and lexicographic max-ordering techniques \cite{Ehrgott2005Multicriteria}. 

In this paper, we consider a new constrained decision problem to find absolute ratings of alternatives that are compared in pairs under multiple criteria, subject to constraints in the form of two-sided bounds (box-constraints) on ratios between the ratings. Given matrices of pairwise comparisons according to the criteria, the problem is formulated as the log-Chebyshev approximation of these matrices by a common consistent matrix (a symmetrically reciprocal matrix of unit rank) to minimize the approximation errors for all matrices simultaneously \cite{Krivulin2019Tropical}. We rearrange the approximation problem as a constrained multiobjective optimization problem of finding a vector that determines the approximating consistent matrix. The constrained optimization problem is then represented in the framework of tropical algebra where addition is defined as taking maximum and multiplication is as usual.

We apply results and further develop methods of tropical optimization to handle the multiobjective optimization problem according to various principles of optimality. The results obtained include several statements proved as lemmas and theorems to provide a basis for formal derivation and justification of the solution procedures. We present numerical examples of solving multicriteria problems of rating four alternatives to illustrate the technique. Specifically, we show how the proposed techniques is used to solve the problem of selecting a plan for vacation from \cite{Saaty1977Scaling}, and compare results.

The novelty of this study is twofold and lies in the development of new solution methods and results that are important for both decision making and tropical optimization. First, we consider a new class of constrained multicriteria pairwise comparison problems where the individual ratings of alternatives are evaluated subject to box-constraints on the ratios of ratings. These constraints can describe predefined relations between ratings, derived from some prior knowledge (such as approved and adopted results of previous studies) independent of the current comparison data. We extend the solution obtained in \cite{Krivulin2021Algebraic} for a constrained bi-criteria problem to the problems with arbitrary fixed number of criteria. The constrained problems under study allow to take into account the increasing complexity of contemporary decision processes, yet they have received little or no attention in the literature. 

Furthermore, we develop a new technique to handle the above constrained problems, which combines the log-Chebyshev approximation of pairwise comparison matrices with multicriteria optimization based on the max-ordering, lexicographic ordering and lexicographic max-ordering principles of optimality. The solutions are given analytically in a compact parametric form that represents all solution vectors of ratings in a way ready for formal analysis and efficient computation. We observe that implementation of the lexicographic ordering and lexicographic max-ordering optimization schemes to solve multiobjective problems involves solving a series of constrained scalar optimization problems, where the constraints of a problem define a feasible set as the solution set of the previous problem in the series \cite{Ehrgott2005Multicriteria}. The need for this predetermines and justifies the application of the log-Chebyshev approximation which can solve constrained problems analytically, rather than the methods of the principal eigenvector and geometric mean, which are not able to address the constrained pairwise comparison problems in a direct way. 

Second, we present new results on the solution of tropical optimization problems, which are of independent interest. A theorem is derived that formally establishes the equivalence between the solution of a constrained tropical optimization problem, the solution of a linear vector inequality and a set of vectors given in parametric form. The theorem is obtained as a consequence of some known facts (see, e.g. \cite{Krivulin2015Extremal,Krivulin2015Multidimensional,Krivulin2017Direct,Krivulin2017Tropical}) combined into one result here for the first time to serve as a useful analytical tool for implementation in tropical optimization. Specifically, this result plays a key role in the construction and justification of solution procedures for the constrained multicriteria pairwise comparison problem under consideration. 

We derive new solutions for particular problems of maximizing and minimizing a function in the multiplicative form of the Hilbert seminorm under a normalization constraint. This result is used to deal with multiple solutions of log-Chebyshev approximation in the pairwise comparison problem, which can make it difficult to choose the optimal solution in decision making. We improve and simplify the technique developed in \cite{Krivulin2019Tropical} to handle the multiple solutions issue by finding two solution vectors that have the highest ratio (the best differentiating solution) and the lowest ratio (the worst differentiating solution) between ratings. For some problems, the technique may fail to overcome the issue and yield both best and worst differentiating solutions that are not unique. We propose a new approach, which in case of multiple solutions allows to reduce the number of best differentiating solutions and replace all worst differentiating solutions by one solution.

Finally, we develop new computational procedures to solve a class of constrained multiobjective tropical optimization problems according to various principles of optimality, which can find application in other practical contexts including temporal project scheduling and facility location problems.

The rest of the paper is organized as follows. We start with Section~\ref{S-SCPCP} where a constrained pairwise comparison problem with a single criterion and its solution based on the log-Chebyshev approximation are considered. In Section~\ref{S-MPCP}, we extend the pairwise comparison problem to take into account multiple criteria and describe various solution approaches. Section~\ref{S-PADNR} offers an outline of basic definitions and results of tropical max-algebra needed below. In Section~\ref{S-TOP}, we present solutions to tropical optimization problems, which are used to handle multicriteria pairwise comparison problems in Section~\ref{S-SMPCP}. Sections~\ref{S-IE} and \ref{S-SVPP} give numerical examples to illustrate the obtained results. Finally, Section~\ref{S-C} includes short concluding remarks.

\section{Single-Criterion Pairwise Comparison Problem}
\label{S-SCPCP}

Suppose that $n$ alternatives are compared in pairs, which results in an $(n\times n)$-matrix $\bm{C}=(c_{ij})$ of pairwise comparisons, where the entry $c_{ij}>0$ shows by how many times alternative $i$ is considered more preferable than alternative $j$. The pairwise comparison matrix $\bm{C}$ is assumed to satisfy the condition $c_{ij}=1/c_{ji}$ (and thus $c_{ii}=1$) to be valid for all $i,j=1,\ldots,n$, which makes the matrix $\bm{C}$ be symmetrically reciprocal. This condition says that if alternative $i$ is estimated to be $c_{ij}$ times better than $j$, then alternative $j$ must be $1/c_{ij}$ times better ($c_{ij}$ times worse) than $i$.

Given a pairwise comparison matrix $\bm{C}$ which represents the results of relative evaluation of one alternative against another, the problem of interest is to calculate absolute ratings of alternatives in the form of an $n$-vector $\bm{x}=(x_{j})$ where $x_{j}>0$ represents the rating of alternative $j$.

A pairwise comparison matrix $\bm{C}$ is referred to as consistent if the condition $c_{ij}=c_{ik}c_{kj}$ holds for all $i,j,k=1,\ldots,n$. This condition corresponds to the natural transitivity of judgments, which requires that if alternative $i$ is considered $c_{ik}$ times better than $k$, and alternative $k$ is $c_{kj}$ times better than $j$, then alternative $i$ should be $c_{ik}c_{kj}$ times better than $j$. 

For a consistent matrix $\bm{C}$, it is not difficult to verify that there exists a positive vector $\bm{x}=(x_{j})$ whose entries determine the entries of $\bm{C}$ by the relation $c_{ij}=x_{i}/x_{j}$ to be valid for all $i,j$, which means that the matrix $\bm{C}$ is of unit rank. Indeed, the transitivity condition $c_{ij}=c_{ik}c_{kj}$ implies that any two columns $j$ and $k$ in $\bm{C}$ are collinear, and thus we can write $\bm{C}=\bm{x}\bm{y}^{T}$ where $\bm{x}=(x_{j})$ and $\bm{y}=(y_{j})$ are positive column vectors.

Since each entry of the matrix $\bm{C}$ is given by $c_{ij}=x_{i}y_{j}$ with $x_{i}y_{i}=c_{ii}=1$, we have $y_{i}=1/x_{i}$ for all $i$, which yields the above relation. Moreover, it directly follows from this relation that the vector $\bm{x}$, which is defined up to a positive factor, can be taken as a vector of absolute ratings of alternatives and thus gives the solution of the pairwise comparison problem.

The matrices of pairwise comparisons that appear in real-world problems are commonly not consistent, which makes the problem of evaluating absolute ratings nontrivial. The solution techniques available to handle the problem include various heuristic methods that offer acceptable results in many cases in practice, and approximation methods that provide mathematically justified optimal solutions.

\subsection{Solution Approaches to Pairwise Comparison Problem}

The heuristic methods are mainly based on aggregating columns of the pairwise comparison matrix to obtain a solution by calculating a weighted sum of these columns \cite{Choo2004Common}. In the most widely used method of principal eigenvector \cite{Saaty1977Scaling,Saaty1990Analytic,Saaty2013Onthemeasurement}, the solution vector $\bm{x}$ is defined as the sum of columns taken with weights assumed proportional to the components of $\bm{x}$. Under this assumption, the vector $\bm{x}$ must satisfy the equation $\bm{C}\bm{x}=\lambda\bm{x}$ where $1/\lambda$ is a coefficient of proportionality, and it is found as the principal (Perron) eigenvector of the matrix $\bm{C}$ corresponding to the maximum eigenvalue.

The approximation methods reduce the problem to finding a consistent matrix $\bm{X}$ that approximates a given matrix $\bm{C}$ in the sense of some distance function as approximation error \cite{Choo2004Common}. Then, a positive vector $\bm{x}$ that determines the approximating consistent matrix is taken as a vector of absolute ratings of alternatives. The approximation problem is solved as an optimization problem of minimizing the distance between the matrices $\bm{C}$ and $\bm{X}$, which provides a formal justification of the solution obtained.

If the approximation error is measured on the standard linear scale, the approach normally leads to complex multiextremal optimization problems that are very hard to solve \cite{Saaty1984Comparison}. On the contrary, the application of the logarithmic scale (with logarithm to a base greater than $1$) makes the approximation problem more tractable and even allows the solution to be derived analytically in an exact explicit form.

A widespread approximation technique to solve the pairwise comparison problem measures the error between the matrices $\bm{C}=(c_{ij})$ and $\bm{X}=(x_{i}/x_{j})$ by using the Euclidean metric in logarithmic scale \cite{Narasimhan1982Geometric,Crawford1985Note,Barzilai1987Consistent}. This technique known as the method of geometric mean, involves finding a positive vector $\bm{x}=(x_{j})$ that solves the problem 
\begin{equation*}
\begin{aligned}
%&
%\text{minimize}
\min_{\bm{x}>\bm{0}}
%&&
&&&
\sum_{1\leq i,j\leq n}\left(\log c_{ij}-\log\frac{x_{i}}{x_{j}}\right)^{2}.
%\max_{1\leq i,j\leq n}x_{i}^{-1}c_{ij}x_{j}.
\end{aligned}
\end{equation*}

The standard solution approach, which applies the first derivative test to find the roots of the partial derivatives of the objective function with respect to all $x_{i}$, yields a solution vector $\bm{x}$ with the components given in the parametric form
\begin{equation*}
x_{i}
=
\left(
%\prod_{1\leq j\leq n}a_{ij}
\prod_{j=1}^{n}c_{ij}
\right)^{1/n}
u,
\qquad
u>0,
\qquad
i=1,\ldots,n.
\end{equation*}

After normalization (e.g. by dividing by the sum of components), this vector is taken as the optimal solution of the pairwise comparison problem.

As another approximation technique, a method of minimizing the Chebyshev distance in logarithmic scale (a log-Chebyshev approximation method) is proposed in \cite{Krivulin2015Rating,Krivulin2016Using,Krivulin2019Tropical}. The method aims to find positive vectors $\bm{x}$ that solve the problem
\begin{equation}
\begin{aligned}
%&
%\text{minimize}
\min_{\bm{x}>\bm{0}}
%&&
&&&
\max_{1\leq i,j\leq n}\left|\log c_{ij}-\log\frac{x_{i}}{x_{j}}\right|.
%\max_{1\leq i,j\leq n}x_{i}^{-1}c_{ij}x_{j}.
\end{aligned}
\label{P-minx_maxijcijlogxixj}
\end{equation}

We observe that the methods of principal eigenvector and geometric mean lead to unique solutions and offer efficient computational procedures of calculating the result. At the same time, the solution provided by log-Chebyshev approximation can be nonunique and thus require further analysis. Considering an approximate character of the pairwise comparison model which usually leads to inconsistent results of pairwise comparisons, multiple solutions of the problem look quite reasonable and even allow to select a solution satisfying additional constraints.
 
%In the next subsection, the log-Chebyshev approximation based solution to the problem with additional constraints on absolute ratings is discussed in more detail.

\subsection{Log-Chebyshev Approximation of Pairwise Comparison Matrix}

Let us consider the problem of unconstrained log-Chebyshev approximation at \eqref{P-minx_maxijcijlogxixj}. Observing that the logarithm to a base greater than $1$ monotonically increases, the objective function is rewritten as (see, e.g. \cite{Krivulin2019Tropical})
\begin{equation*}
\max_{1\leq i,j\leq n}\left|\log c_{ij}-\log\frac{x_{i}}{x_{j}}\right|
=
\log\max_{1\leq i,j\leq n}\frac{c_{ij}x_{j}}{x_{i}}.
\end{equation*}
%\begin{multline*}
%\max_{1\leq i,j\leq n}\left|\log c_{ij}-\log\frac{x_{i}}{x_{j}}\right|
%=
%\max_{1\leq i,j\leq n}\max\left\{\log\frac{c_{ij}x_{j}}{x_{i}},\log\frac{x_{i}}{c_{ij}x_{j}}\right\}
%\\
%=
%\log\max_{1\leq i,j\leq n}\max\left\{\frac{c_{ij}x_{j}}{x_{i}},\frac{a_{ji}x_{i}}{x_{j}}\right\}
%=
%\log\max_{1\leq i,j\leq n}\frac{c_{ij}x_{j}}{x_{i}}.
%\end{multline*}

Since the logarithmic function on the right-hand side attains its maximum there where its argument is maximal, we remove the logarithm from the objective function. Moreover, it is not difficult to verify (see \cite{Elsner2004Maxalgebra,Elsner2010Maxalgebra,Krivulin2019Tropical}) that the obtained objective function is connected with the standard relative approximation error by the equality
\begin{equation*}  
\max_{1\leq i,j\leq n}\frac{c_{ij}x_{j}}{x_{i}}
=
\max_{1\leq i,j\leq n}\frac{|c_{ij}-x_{i}/x_{j}|}{c_{ij}}
+
1.
\end{equation*}
%\begin{multline*}  
%\max_{1\leq i,j\leq n}\frac{|c_{ij}-x_{i}/x_{j}|}{c_{ij}}
%%=
%%\max_{1\leq i,j\leq n}\left|\frac{x_{i}}{c_{ij}x_{j}}-1\right|
%=
%\max_{i<j}\max\left\{\left|\frac{x_{i}}{c_{ij}x_{j}}-1\right|,\left|\frac{c_{ij}x_{j}}{x_{i}}-1\right|\right\}
%\\
%=
%\max_{i<j}\max\left\{\frac{x_{i}}{c_{ij}x_{j}}-1,\frac{c_{ij}x_{j}}{x_{i}}-1\right\}
%%=
%%\max_{i<j}\max\left\{\frac{x_{i}}{c_{ij}x_{j}},\frac{c_{ij}x_{j}}{x_{i}}\right\}-1
%=
%\max_{1\leq i,j\leq n}\frac{c_{ij}x_{j}}{x_{i}}
%-
%1.
%\end{multline*}  
This means that the log-Chebyshev approximation of the matrix $\bm{C}$ is equivalent to the approximation in the sense of relative error as well.

Consider an extension of the pairwise comparison problem where the ratings of alternatives are to be evaluated subject to some predefined constraints. Suppose the constraints imposed on the ratings are in the form of two-sided bounds on ratios between the ratings. These constraints can reflect prior knowledge about the relationship between ratings, which can be obtained by different assessment techniques or resulted from the nature of alternatives.

Given a nonnegative matrix $\bm{B}=(b_{ij})$ where $b_{ij}>0$ shows that alternative $i$ must be considered not less than $b_{ij}$ times better than $j$, and $b_{ij}=0$ indicates that no constraint is imposed on $i$ with respect to $j$. The constraints are given by the inequalities $b_{ij}x_{j}\leq x_{i}$ for all $i,j=1,\ldots,n$, or equivalently by the double inequality $b_{ij}\leq x_{i}/x_{j}\leq1/b_{ji}$ when $b_{ji}>0$. Combining the former inequalities for all $j$ yields the system of inequalities 
\begin{equation*}
\max_{1\leq j\leq n}b_{ij}x_{j}
\leq
x_{i},
\qquad
i=1,\ldots,n.
\end{equation*}

We incorporate the constraints into \eqref{P-minx_maxijcijlogxixj} to formulate the following problem of constrained log-Chebyshev approximation. Given a positive $(n\times n)$-matrix $\bm{C}=(c_{ij})$ of pairwise comparisons and nonnegative $(n\times n)$-matrix $\bm{B}=(b_{ij})$ of constraints, the problem is to find positive $n$-vectors $\bm{x}=(x_{j})$ that achieve
\begin{equation}
\begin{aligned}
%&
%\text{minimize}
\min_{\bm{x}>\bm{0}}
%&&
&&&
\max_{1\leq i,j\leq n}\frac{c_{ij}x_{j}}{x_{i}};
%\max_{1\leq i,j\leq n}x_{i}^{-1}c_{ij}x_{j};
\\
%&
%\text{subject to}
\text{s.t.}
%&&
&&&
\max_{1\leq j\leq n}b_{ij}x_{j}
\leq
x_{i},
\qquad
i=1,\ldots,n.
\end{aligned}
\label{P-minx_maxijcijxixj-maxjbijxjleqxi}
\end{equation}

Note that multiplying all components of the vector $\bm{x}$ by a common positive factor does not change both objective function and inequality constraints in problem \eqref{P-minx_maxijcijxixj-maxjbijxjleqxi}, and hence the solutions of the problem are scale-invariant.

Finally, we observe that the methods of principal eigenvector and geometric mean cannot accommodate the above constraints in a straightforward way without considerable modification and complication of the solution. As it is shown later, the constrained problem of log-Chebyshev approximation at \eqref{P-minx_maxijcijxixj-maxjbijxjleqxi} can be directly solved to give the result in a compact vector form.

\subsection{Best and Worst Differentiating Solutions}

If problem \eqref{P-minx_maxijcijxixj-maxjbijxjleqxi} has a unique solution up to a positive factor, this solution is taken as the vector of absolute ratings of alternatives in question.

Suppose that the solution of \eqref{P-minx_maxijcijxixj-maxjbijxjleqxi} is not unique and denote the set of obtained solution vectors $\bm{x}$ by
\begin{equation*}
X
=
\arg\min_{\bm{x}>\bm{0}}
\left\{
\max_{1\leq i,j\leq n}\frac{c_{ij}x_{j}}{x_{i}}\
:\
\max_{1\leq j\leq n}b_{ij}x_{j}
\leq
x_{i},
\quad
i=1,\ldots,n
\right\}.
\end{equation*}
%\begin{equation*}
%S
%=
%\arg\min_{\bm{x}}
%\left\{
%\max_{1\leq i,j\leq n}x_{i}^{-1}c_{ij}x_{j}\
%:\
%\max_{1\leq j\leq n}b_{ij}x_{j}
%\leq
%x_{i},
%\quad
%i=1,\ldots,n
%\right\}.
%\end{equation*}

A nonunique solution vector of the pairwise comparison problem may make it difficult to derive an optimal decision, and thus further analysis is needed to characterize the result by one or two vectors that are reasonably representative of the solution. An approach developed in \cite{Krivulin2019Tropical} suggests to reduce the entire set of solutions to two vectors that provide the ``best'' and ``worst'' answers to the questions of which alternative has the highest rating and how much this rating differs from the ratings of other alternatives. As the best solution, the approach takes a vector that maximally differentiates the alternatives with the highest and lowest ratings, and as the worst a vector that minimally differentiates all alternatives. The difference between alternatives is measured by the ratio in the form of a multiplicative version of the Hilbert (span, range) seminorm of the vector $\bm{x}$, which is given by
\begin{equation}
\max_{1\leq i\leq n}x_{i}
%\left/
\Big/
\min_{1\leq j\leq n}x_{j}
%\right.
=
\max_{1\leq i\leq n}x_{i}
\times
%\max_{1\leq j\leq n}\frac{1}{x_{j}}.
\max_{1\leq j\leq n}x_{j}^{-1}.
\label{E-maxximinxi}
\end{equation}

The best and worst differentiating solutions are then obtained by maximizing and minimizing the Hilbert seminorm over all vectors $\bm{x}\in X$, which leads to the optimization problems
\begin{equation*}
\begin{aligned}
%&
%\text{minimize}
\max_{\bm{x}\in X}
%&&
&&&
%\max_{1\leq i\leq n}x_{i}\times\max_{1\leq j\leq n}\frac{1}{x_{j}},
\max_{1\leq i\leq n}x_{i}\times\max_{1\leq j\leq n}x_{j}^{-1};
\end{aligned}
\qquad\qquad
\begin{aligned}
%&
%\text{minimize}
\min_{\bm{x}\in X}
%&&
&&&
%\max_{1\leq i\leq n}x_{i}\times\max_{1\leq j\leq n}\frac{1}{x_{j}}.
\max_{1\leq i\leq n}x_{i}\times\max_{1\leq j\leq n}x_{j}^{-1}.
\end{aligned}
\end{equation*}

As examples show (see, e.g. \cite{Krivulin2019Tropical}), one of the shortcomings of this approach is that the best (worst) differentiating solution found in this way may not be unique. To overcome this possible issue, we now propose a new improved procedure of finding the best and worst differentiating solutions, which provides a reduced set of the best solutions and a unique worst solution. 

Since the solutions of problem \eqref{P-minx_maxijcijxixj-maxjbijxjleqxi} are invariant under multiplication by a positive factor, we can restrict ourselves to vectors that are normalized by dividing by the maximum component. Adding the normalization condition transforms the last two problems into the problems
\begin{equation}
\begin{aligned}
%&
%\text{minimize}
\max_{\bm{x}\in X}
%&&
&&&
\max_{1\leq i\leq n}x_{i}^{-1};
\\
%&
%\text{subject to}
\text{s.t.}
%&&
&&&
\max_{1\leq i\leq n}x_{i}=1;
\end{aligned}
\qquad\qquad
\begin{aligned}
%&
%\text{minimize}
\min_{\bm{x}\in X}
%&&
&&&
\max_{1\leq i\leq n}x_{i}^{-1};
\\
%&
%\text{subject to}
\text{s.t.}
%&&
&&&
\max_{1\leq i\leq n}x_{i}=1.
\end{aligned}
\label{P-maxmaxi1xi-maxixi1-minmaxi1xi-maxixi1}
\end{equation}

Furthermore, we note that both problems can have nonunique normalized solutions. Under this circumstance, it is reasonable to concentrate only on the minimal normalized solution to the maximization problem and the maximal normalized solution to the minimization problem. As usual, a solution $\bm{x}_{0}$ is called minimal (maximal) if the componentwise inequality $\bm{x}_{0}\leq\bm{x}$ ($\bm{x}_{0}\geq\bm{x}$) holds for any solution $\bm{x}$.

Indeed, all normalized vectors that maximize or minimize the ratio between the highest and lowest ratings have two common components whose ratio is fixed: the maximum component equal to 1, and the minimum component less or equal to 1. In this case, the lower (higher) the ratings of the other alternatives, the better (worse) the alternative with the maximum rating is distinguishable from the others.

\section{Multicriteria Pairwise Comparison Problems}
\label{S-MPCP}

We now assume that $n$ alternatives are compared in pairs according to $m$ criteria. For each criterion $l=1,\ldots,m$, the results of pairwise comparisons are represented by a matrix $\bm{C}_{l}=(c_{ij}^{(l)})$. The problem is to assess the alternatives by evaluating a vector $\bm{x}=(x_{j})$ of ratings subject to constraints given by a matrix $\bm{B}=(b_{ij})$. Application of the log-Chebyshev approximation technique yields the following constrained multiobjective problem:
\begin{equation}
\begin{aligned}
%&
%\text{minimize}
\min_{\bm{x}>\bm{0}}
%&&
&&&
\left(
\max_{1\leq i,j\leq n}\frac{c_{ij}^{(1)}x_{j}}{x_{i}},
\ldots,
\max_{1\leq i,j\leq n}\frac{c_{ij}^{(m)}x_{j}}{x_{i}}
\right);
\\
%&
%\text{subject to}
\text{s.t.}
%&&
&&&
\max_{1\leq j\leq n}b_{ij}x_{j}
\leq
x_{i},
\qquad
i=1,\ldots,n.
\end{aligned}
\label{P-minx_maxijc1ijxixj_maxijcmijxixj-maxbijxjleqxi}
\end{equation}

Suppose we have a solution to the problem, which forms a nonempty set of solution vectors. If the solution obtained is a unique vector up to a positive factor, we normalize this vector by dividing by the maximum element and take its entries as absolute ratings of alternatives in the pairwise comparison problem. Otherwise, the best and worst differentiating vectors of ratings are obtained respectively as the minimal solution of the maximization problem and the maximal solution of the minimization problem at \eqref{P-maxmaxi1xi-maxixi1-minmaxi1xi-maxixi1}.

In the rest of this section, we consider three common approaches to handle problem \eqref{P-minx_maxijc1ijxixj_maxijcmijxixj-maxbijxjleqxi}, which results in the development of new procedures to find the solution set. The proposed solutions follow the max-ordering, lexicographic ordering and lexicographic max-ordering principles of optimality \cite{Ehrgott2005Multicriteria}.

\subsection{Max-Ordering Solution}

The max-ordering optimization aims at minimizing the worst value of the scalar objective functions in the multiobjective problem. For each $\bm{x}$, the approach considers that function which takes the maximal (worst) value. This leads to the replacement of the vector objective function by a scalar function given by the maximum component of the vector function, which is known as the Chebyshev scalarization approach in multiobjective optimization \cite{Nakayama2009Sequential}.

The solution of the constrained problem at \eqref{P-minx_maxijc1ijxixj_maxijcmijxixj-maxbijxjleqxi} starts with the introduction of the feasible solution set $X_{0}$ given by the constraints as
\begin{equation}
X_{0}
=
\left\{\bm{x}>\bm{0}\
:\
\max_{1\leq j\leq n}b_{ij}x_{j}
\leq
x_{i},
\quad
i=1,\ldots,n
\right\}.
\label{E-S0}
\end{equation}

We apply the Chebyshev scalarization and use the associativity of the maximum operation to form the scalar objective function
\begin{equation*}
\max_{1\leq l\leq m}
\max_{1\leq i,j\leq n}\frac{c_{ij}^{(l)}x_{j}}{x_{i}}
=
\max_{1\leq i,j\leq n}\frac{a_{ij}x_{j}}{x_{i}},
\qquad
a_{ij}
=
\max_{1\leq l\leq m}c_{ij}^{(l)}.
\end{equation*}

Then, the problem reduces to the minimization problem
\begin{equation}
\begin{aligned}
%&
%\text{minimize}
\min_{\bm{x}\in X_{0}}
%&&
&&&
\max_{1\leq i,j\leq n}\frac{a_{ij}x_{j}}{x_{i}},
\end{aligned}
\label{P-minx_maxijaijxixj}
\end{equation}
which is solved to obtain the max-ordering solution in the form of the set
\begin{equation*}
X_{1}
=
\arg\min_{\bm{x}\in X_{0}}
\max_{1\leq i,j\leq n}\frac{a_{ij}x_{j}}{x_{i}}.
\end{equation*}

Note that the solution obtained by the max-ordering optimization is known to be week Pareto-optimal and become Pareto-optimal if unique \cite{Nakayama2009Sequential}. 

If the solution is not unique (up to a positive factor), we further reduce the solution set by extracting the best and worst differentiating vectors of ratings given by the minimal and maximal normalized solutions of problems \eqref{P-maxmaxi1xi-maxixi1-minmaxi1xi-maxixi1} over the set $X_{1}$.

\subsection{Lexicographic Ordering Solution}

The lexicographic optimization examines the scalar objective functions in a hierarchical order based on certain ranking of objectives. Suppose the objectives are numbered in such a way that objective $1$ has the highest rank, objective $2$ has the second highest and so on until objective $m$ having the lowest rank. The lexicographic approach first solves the problem of minimizing function $1$ and examines the set of optimal solutions obtained. If the solution is unique, it is taken as the solution of the overall multiobjective problem. Otherwise function $2$ is minimized over all solutions of the first problem, and the procedure continues until a unique solution is obtained or the problem with function $m$ is solved. 

To apply this approach to problem \eqref{P-minx_maxijc1ijxixj_maxijcmijxixj-maxbijxjleqxi}, we first take the set $X_{0}$ given by \eqref{E-S0}, and then successively obtain the solution sets $X_{s}$ for each problem
\begin{equation}
\begin{aligned}
%&
%\text{minimize}
\min_{\bm{x}\in X_{s-1}}
%&&
&&&
\max_{1\leq i,j\leq n}\frac{c_{ij}^{(s)}x_{j}}{x_{i}},
%\max_{1\leq i,j\leq n}x_{i}^{-1}c_{ij}x_{j},
\qquad
s=1,\ldots,m.
\end{aligned}
\label{P-minx_maxijcijsxixj}
\end{equation}

The solution procedure stops at step $s<m$ if the set $X_{s}$ consists of a single solution vector, or at step $s=m$ otherwise. The last found set $X_{s}$ is taken as the lexicographic solution for the entire problem. In the same way as before, if the solution given by $X_{s}$ is not unique, it is reduced to the minimal and maximal solutions of respective problems at \eqref{P-maxmaxi1xi-maxixi1-minmaxi1xi-maxixi1} over $X_{s}$.

\subsection{Lexicographic Max-Ordering Solution}

This approach combines the lexicographic ordering and max-ordering into one procedure that improves the accuracy of the assessment provided by the max-ordering approach. The procedure consists of several steps, each of which finds the max-ordering solution of a reduced problem that has a lower multiplicity of objectives and smaller feasible set. The first solution step coincides with the above described max-ordering solution of the constrained problem with $m$ objectives and the feasible solution set given by the constraints. Each subsequent step takes the solution from the previous step as the current feasible set and selects objectives that can be further minimized over this set, to incorporate into a current vector objective function. A scalar objective function is included if it has its minimum value over the current feasible set below the minimum of the objective function at the previous step. 

We denote by $I_{s}$ the set of indices of scalar objective functions that form components of the vector objective function, and by $X_{s}$ the solution set obtained at step $s$. We initially set $I_{0}=\{1,\ldots,m\}$ and define $X_{0}$ as in \eqref{E-S0}.

We find a solution set $X_{s}$ by solving the problem
\begin{equation}
\begin{aligned}
%&
%\text{minimize}
\min_{\bm{x}\in X_{s-1}}
%&&
&&&
\max_{l\in I_{s-1}}
\max_{1\leq i,j\leq n}\frac{c_{ij}^{(l)}x_{j}}{x_{i}},
\qquad
s=1,\ldots,m.
\end{aligned}
\label{P-minx_maxkl1maxijcijlxixj}
\end{equation}

With the minimum value of the objective function at step $s$ denoted by $\theta_{s}$, we define the index set
\begin{equation}
I_{s}
=
\left\{
l\in I_{s-1}
:\
\theta_{s}
>
\min_{\bm{x}\in X_{s}}
\max_{1\leq i,j\leq n}\frac{c_{ij}^{(l)}x_{j}}{x_{i}}
\right\}.
\label{E-Is}
\end{equation}

The procedure is completed at step $s<m$ if either the set $X_{s}$ reduces to a single solution vector or the condition $I_{s}=\emptyset$ holds. We additionally solve optimization problems \eqref{P-maxmaxi1xi-maxixi1-minmaxi1xi-maxixi1} if the final set $X_{s}$ provides a nonunique solution.

Below we show how the solutions offered by the above three approaches can be directly represented in explicit analytical form using methods and result of tropical mathematics.

\section{Preliminary Algebraic Definitions, Notation and Results}
\label{S-PADNR}

In this section, we offer a brief overview of definitions, notation and results of tropical algebra that are used in the sequel. For further detail on the theory and application of tropical mathematics, one can consult the resent monographs and textbooks \cite{Golan2003Semirings,Heidergott2006Maxplus,Gondran2008Graphs,Maclagan2015Introduction}. 

\subsection{Idempotent Semifield}

Consider the set $\mathbb{R}_{+}$ of nonnegative reals with two binary operations: addition denoted by $\oplus$ and defined as maximum, and multiplication denoted and defined as usual. Both operations are associative and commutative, and multiplication distributes over addition. Addition is idempotent since $x\oplus x=\max(x,x)=x$ for all $x\in\mathbb{R}_{+}$, and thus not invertible. It induces a partial order by the rule: $x\leq y$ if and only if $x\oplus y=y$, which is in agreement with the natural linear order on $\mathbb{R}_{+}$. With these properties, the system $\langle\mathbb{R}_{+},0,\oplus\rangle$ forms an idempotent Abelian semigroup with identity.

With usual multiplication, the system $\langle\mathbb{R}_{+},1,\times\rangle$ has the structure of an Abelian group where powers with rational (and even real) exponents are well defined, which means that this system is radicable (algebraically complete). In what follows, the multiplication sign $\times$ is omitted as in the standard algebra to save writing. Under the above properties, the system $\langle\mathbb{R}_{+},0,1,\oplus,\times\rangle$ is commonly classified as a linearly ordered idempotent radicable semifield and referred to as max-algebra.

Both addition and multiplication are monotone in each argument: if the inequality $x\leq y$ holds for some $x,y\in\mathbb{R}_{+}$, then the inequalities $x\oplus z\leq y\oplus z$ and $xz\leq yz$ are valid for any $z\in\mathbb{R}_{+}$ as well. Furthermore, it results from the inequality $x\leq y$ with $x,y\ne0$ that $x^{q}\geq y^{q}$ if $q\leq0$ and $x^{q}\leq y^{q}$ if $q>0$.

\subsection{Matrix and Vector Algebra}

Matrices over max-algebra are introduced in the usual way; addition and multiplication of compatible matrices as well as multiplication of matrices by scalars follow the standard entrywise rules where the arithmetic addition $+$ is replaced by $\oplus=\max$ (whereas the multiplication does not change). The monotonicity of scalar operations extends to the matrix operations, where the inequalities are understood entrywise.

The transpose of a matrix $\bm{A}$ is denoted by $\bm{A}^{T}$. The matrices which consist of one column are vectors, and of one row are transposed vectors. The set of matrices of $m$ rows and $n$ columns is denoted by $\mathbb{R}_{+}^{m\times n}$, and the set of column vectors with $n$ entries is by $\mathbb{R}_{+}^{n}$.

The zero matrix denoted by $\bm{0}$ and the identity matrix denoted by $\bm{I}$ are defined in the same way as in conventional algebra. A matrix without zero rows (columns) is called row-regular (column-regular), and a vector without zero entries is called regular. The matrices without zero entries and regular vectors are also called positive.

A vector $\bm{y}$ is collinear to vector $\bm{x}$, if $\bm{y}=c\bm{x}$ for some $c\in\mathbb{R}_{+}$.

For any nonzero matrix $\bm{A}=(a_{ij})$ from the set $\mathbb{R}_{+}^{m\times n}$, its multiplicative conjugate transpose (or simply conjugate) is the matrix $\bm{A}^{-}=(a_{ij}^{-})$ in $\mathbb{R}_{+}^{n\times m}$ where $a_{ij}^{-}=a_{ji}^{-1}$ if $a_{ji}\ne0$, and $a_{ij}^{-}=0$ otherwise. If a matrix $\bm{A}$ is row-regular, then the following inequality obviously holds:
\begin{equation*}
\bm{A}\bm{A}^{-}
\geq
\bm{I}.
\end{equation*}

For any nonzero vector $\bm{x}=(x_{j})$ from $\mathbb{R}_{+}^{n}$, its conjugate is the row vector $\bm{x}^{-}=(x_{j}^{-})$ where $x_{j}^{-}=x_{j}^{-1}$ if $x_{j}\ne0$, and $x_{j}^{-}=0$ otherwise. If a vector $\bm{x}$ is regular, then the following matrix inequality and scalar equality are valid:
\begin{equation*}
\bm{x}\bm{x}^{-}
\geq
\bm{I},
\qquad
\bm{x}^{-}\bm{x}
=
1
\end{equation*}
(where the equality holds for any nonzero vector $\bm{x}$).

For any square matrix $\bm{A}=(a_{ij})$ from $\mathbb{R}_{+}^{n\times n}$, the nonnegative integer powers represent repeated multiplication of the matrix by itself, defined by $\bm{A}^{0}=\bm{I}$ and $\bm{A}^{p}=\bm{A}\bm{A}^{p-1}$ for all integer $p>0$. The trace of $\bm{A}$ is given by
\begin{equation*}
\mathop\mathrm{tr}\bm{A}
=
a_{11}\oplus\cdots\oplus a_{nn}
=
\bigoplus_{k=1}^{n}a_{kk}.
\end{equation*}

For any matrix $\bm{A}\in\mathbb{R}_{+}^{n\times n}$, a trace function is introduced to serve as a tropical analogue of the matrix determinant in the form
\begin{equation*}
\mathop\mathrm{Tr}(\bm{A})
=
\mathop\mathrm{tr}\bm{A}
\oplus\cdots\oplus
\mathop\mathrm{tr}\bm{A}^{n}
=
\bigoplus_{k=1}^{n}\mathop\mathrm{tr}\bm{A}^{k}.
\end{equation*}

If the condition $\mathop\mathrm{Tr}(\bm{A})\leq1$ holds, then the Kleene star operator is defined to map the matrix $\bm{A}$ into the matrix
\begin{equation*}
\bm{A}^{\ast}
=
\bm{I}\oplus\bm{A}\oplus\cdots\oplus\bm{A}^{n-1}
=
\bigoplus_{k=0}^{n-1}\bm{A}^{k}.
\end{equation*}

Moreover, under the same condition, the inequality
\begin{equation*}
\bm{A}^{\ast}
\geq
\bm{A}^{p},
\end{equation*}
is valid for all integer $p\geq0$, from which it follows directly that
\begin{equation*}
\bm{A}^{\ast}\bm{A}^{\ast}
=
\bm{A}^{\ast}.
\end{equation*}

The spectral radius of a matrix $\bm{A}\in\mathbb{R}_{+}^{n\times n}$ is given by
\begin{equation*}
\lambda
=
\mathop\mathrm{tr}\bm{A}
\oplus\cdots\oplus
\mathop\mathrm{tr}\nolimits^{1/n}(\bm{A}^{n})
=
\bigoplus_{k=1}^{n}
\mathop\mathrm{tr}\nolimits^{1/k}(\bm{A}^{k}).
\end{equation*}

Consider a matrix $\bm{A}=(a_{ij})$ in $\mathbb{R}_{+}^{m\times n}$ and vector $\bm{x}=(x_{j})$ in $\mathbb{R}_{+}^{n}$. With the notation $\bm{1}=(1,\ldots,1)^{T}$, tropical analogues of the matrix and vector norms are defined as 
\begin{equation*}
\|\bm{A}\|
=
\bm{1}^{T}\bm{A}\bm{1}
=
\bigoplus_{i=1}^{m}\bigoplus_{j=1}^{n}a_{ij},
\qquad
\|\bm{x}\|
=
\bm{1}^{T}\bm{x}
=
\bm{x}^{T}\bm{1}
=
\bigoplus_{j=1}^{n}x_{j}.
\end{equation*}

\subsection{Vector Inequalities}

We conclude the overview with two inequalities that appear in the derivation of the solution of optimization problems in what follows. Suppose $\bm{A}\in\mathbb{R}_{+}^{m\times n}$ is a given matrix and $\bm{d}\in\mathbb{R}_{+}^{m}$ is a given vector, and consider the problem to find all vectors $\bm{x}\in\mathbb{R}_{+}^{n}$ that satisfy the inequality
\begin{equation}
\bm{A}\bm{x}
\leq
\bm{d}.
\label{I-Axleqd}
\end{equation}

Solutions of \eqref{I-Axleqd} are known under various assumptions in different forms. We use a solution suggested by the following assertion (see, e.g. \cite{Krivulin2015Extremal}).
\begin{lemma}
\label{L-Axleqd}
For any column-regular matrix $\bm{A}$ and regular vector $\bm{d}$, all solutions of inequality \eqref{I-Axleqd} are given by the inequality $\bm{x}\leq(\bm{d}^{-}\bm{A})^{-}$.
%\begin{equation*}
%\bm{x}
%\leq
%(\bm{d}^{-}\bm{A})^{-}.
%\end{equation*}
\end{lemma}

As a consequence of the lemma, the solution of the equation $\bm{A}\bm{x}=\bm{d}$, if exists, can be represented as a family of solutions each defined by the vector inequality $\bm{x}\leq(\bm{d}^{-}\bm{A})^{-}$ where one scalar inequality is replaced by an equality.

Assume that for a given square matrix $\bm{B}\in\mathbb{R}_{+}^{n\times n}$, we need to find all regular solutions $\bm{x}\in\mathbb{R}_{+}^{n}$ to the inequality 
\begin{equation}
\bm{B}\bm{x}
\leq
\bm{x}.
\label{I-Bxleqx}
\end{equation}

A complete solution of \eqref{I-Bxleqx} can be obtained as follows \cite{Krivulin2015Extremal,Krivulin2015Multidimensional}.
\begin{lemma}
\label{L-Bxleqx}
For any matrix $\bm{B}$, the following statements are true.
\begin{enumerate}
\item
If $\mathop\mathrm{Tr}(\bm{B})\leq1$, then all regular solutions of inequality \eqref{I-Bxleqx} are given in parametric form by $\bm{x}=\bm{B}^{\ast}\bm{u}$ where $\bm{u}\ne\bm{0}$ is a vector of parameters.
\item
If $\mathop\mathrm{Tr}(\bm{B})>1$, then there is no regular solution.
\end{enumerate}
\end{lemma}

\section{Tropical Optimization Problems}
\label{S-TOP}

We are now concerned with optimization problems formulated and solved in the framework of tropical algebra to provide the basis for the solutions of multicriteria pairwise comparison problems in what follows. We present a new result that establishes a direct correspondence between the solutions of a conjugate quadratic optimization problem and the solutions of a linear inequality. Next, we offer new direct solutions to the problems of both maximization and minimization of a function in the form of the Hilbert seminorm.

\subsection{Conjugate Quadratic Optimization Problem}

Let us consider a problem to minimize a conjugate quadratic (or pseudoquadratic) vector form subject to a linear constraint. Given matrices $\bm{A},\bm{B}\in\mathbb{R}_{+}^{n\times n}$, we find regular vectors $\bm{x}\in\mathbb{R}_{+}^{n}$ that attain the minimum
\begin{equation}
\begin{aligned}
%&
%\text{minimize}
\min_{\bm{x}>\bm{0}}
%&&
&&&
\bm{x}^{-}\bm{A}\bm{x};
\\
%&
%\text{subject to}
\text{s.t.}
&&&
\bm{B}\bm{x}
\leq
\bm{x}.
\label{P-minxAx-Bxleqx}
\end{aligned}
\end{equation}

Complete solutions to the problem and its variations are proposed in \cite{Krivulin2015Extremal,Krivulin2015Multidimensional,Krivulin2017Direct,Krivulin2017Tropical}. We use a solution in the form of the next statement.
\begin{theorem}
\label{T-minxAx-Bxleqx}
Let $\bm{A}$ be a matrix with nonzero spectral radius and $\bm{B}$ be a matrix such that $\mathop\mathrm{Tr}(\bm{B})\leq1$. Then, the minimum value of the objective function in problem \eqref{P-minxAx-Bxleqx} is equal to
\begin{equation}
\theta
=
\bigoplus_{k=1}^{n}
\bigoplus_{0\leq i_{1}+\cdots+i_{k}\leq n-k}
\mathop\mathrm{tr}\nolimits^{1/k}(\bm{A}\bm{B}^{i_{1}}\cdots\bm{A}\bm{B}^{i_{k}}),
\label{E-thetaeqtr1nABi1ABik}
\end{equation}
and all regular solutions of the problem are given in the parametric form
\begin{equation}
\bm{x}
=
\bm{G}\bm{u},
%\bm{G}(\theta)\bm{u},
\qquad
\bm{G}
%\bm{G}(\theta)
=
(\theta^{-1}\bm{A}\oplus\bm{B})^{\ast},
\label{E-xeqtheta1ABastu}
\end{equation}
where $\bm{u}\ne\bm{0}$ is a vector of parameters.
\end{theorem}

We note that the above solution has a polynomial computational complexity in the dimension $n$. Specifically, it is shown in \cite{Krivulin2017Direct,Krivulin2017Tropical} that this solution can be calculated with at most $O(n^{5})$ scalar operations.

We now formulate an equivalence statement that allows all solutions of problem \eqref{P-minxAx-Bxleqx} to be described using solutions of a vector inequality in the form of \eqref{I-Bxleqx} and vice versa. This new result plays a key role in the development of solutions for the multiobjective optimization problems in subsequent sections.
\begin{theorem}
\label{T-minxAx-Bxleqx-equivalence}
Let $\bm{A}$ be a matrix with nonzero spectral radius, $\bm{B}$ a matrix such that $\mathop\mathrm{Tr}(\bm{B})\leq1$, and $\theta$ be the minimum value of the objective function in problem \eqref{P-minxAx-Bxleqx}, given by \eqref{E-thetaeqtr1nABi1ABik}. 
Then, the following assertions are equivalent:
\renewcommand{\labelenumi}{(\theenumi)}%
\renewcommand{\theenumi}{\roman{enumi}}%
\begin{enumerate}
\item
$\bm{x}$ is a regular solution of problem \eqref{P-minxAx-Bxleqx}.
\item
$\bm{x}$ is a regular solution of the inequality
\begin{equation*}
(\theta^{-1}\bm{A}\oplus\bm{B})\bm{x}
\leq
\bm{x}.
\end{equation*}
\item
$\bm{x}$ is a regular vector given in the parametric form
\begin{equation*}
\bm{x}
=
(\theta^{-1}\bm{A}\oplus\bm{B})^{\ast}\bm{u},
\end{equation*}
where $\bm{u}\ne\bm{0}$ is a vector of parameters.
\end{enumerate}
\end{theorem}
\begin{proof}
The proof of this equivalence theorem is readily obtained by combining the result of Theorem~\ref{T-minxAx-Bxleqx} with that of Lemma~\ref{L-Bxleqx}.
\end{proof}

\subsection{Maximization and Minimization of Hilbert Seminorm}

In this subsection, we concentrate on optimization problems with the objective function in the form of Hilbert seminorm, which appear in the analysis of multiple solutions of the pairwise comparison problem. We consider constrained problems of maximization and minimization of Hilbert seminorm and present results that allow to find all solutions of the problems.

We start with a conventional form given by \eqref{E-maxximinxi} for the Hilbert seminorm of a vector $\bm{x}=(x_{i})$. After rewriting it in terms of max-algebra, we obtain 
\begin{equation*}
\bigoplus_{i=1}^{n}x_{i}
\bigoplus_{j=1}^{n}x_{j}^{-1}
=
\bm{1}^{T}\bm{x}\bm{x}^{-}\bm{1}
=
\bm{x}^{-}\bm{1}\bm{1}^{T}\bm{x}
=
\|\bm{x}\|\|\bm{x}^{-}\|.
\end{equation*}

We examine optimization problems with $\bm{x}^{-}\bm{1}\bm{1}^{T}\bm{x}$ as the objective function, which we refer to as a tropical representation of the multiplicative Hilbert seminorm or simply as Hilbert seminorm.

First, we suppose that given a matrix $\bm{B}\in\mathbb{R}_{+}^{n\times n}$, we seek for regular solutions $\bm{x}\in\mathbb{R}_{+}^{n}$ of the constrained maximization problem
\begin{equation}
\begin{aligned}
%&
%\text{maximize}
\max_{\bm{x}>\bm{0}}
%&&
&&&
\bm{x}^{-}\bm{1}\bm{1}^{T}\bm{x};
\\
%&
%\text{subject to}
\text{s.t.}
&&&
\bm{B}\bm{x}
\leq
\bm{x}.
\end{aligned}
\label{P-maxxx11x-Bxleqx}
\end{equation}

A solution to problem \eqref{P-maxxx11x-Bxleqx} can be obtained as follows (see \cite{Krivulin2019Tropical}).
\begin{lemma}
\label{L-maxx11x-Bxleqx}
For any matrix $\bm{B}$ without zero entries and with $\mathop\mathrm{Tr}(\bm{B})\leq1$, the maximum in problem \eqref{P-maxxx11x-Bxleqx} is equal to $\|\bm{B}^{\ast}(\bm{B}^{\ast})^{-}\|$, and all regular solutions of the problem are given in the parametric form
\begin{equation*}
\bm{x}
=
\bm{G}\bm{u},
\qquad
\bm{G}
=
\bm{B}^{\ast}(\bm{B}_{lk}^{\ast})^{-}\bm{B}^{\ast}
\oplus
\bm{B}^{\ast},
\qquad
\bm{u}
\ne
\bm{0},
\end{equation*}
where $\bm{B}_{lk}^{\ast}$ is a matrix obtained from the matrix $\bm{B}^{\ast}=(\bm{b}_{j}^{\ast})$ with columns $\bm{b}_{j}^{\ast}=(b_{ij}^{\ast})$ by setting to $0$ all entries except the entry $b_{lk}^{\ast}$ with indices given by
%\begin{equation*}
%k
%=
%\arg\max_{j}\bm{1}^{T}\bm{b}_{j}^{\ast}(\bm{b}_{j}^{\ast})^{-}\bm{1},
%\qquad
%l
%=
%\arg\min_{i}b_{ik}^{\ast}.
%\end{equation*}
\begin{equation*}
k
=
\arg\max_{j}\|\bm{b}_{j}^{\ast}\|\|(\bm{b}_{j}^{\ast})^{-}\|,
\qquad
l
=
\arg\min_{i}b_{ik}^{\ast}.
\end{equation*}
\end{lemma}

It is not difficult to see that the most computationally intensive component of the solution is the calculation of the Kleene star matrix $\bm{B}^{\ast}$. This calculation can take no more that $O(n^{4})$ scalar operations, which determines the computational complexity involved in the solution. 

Another problem of interest is to find regular vectors $\bm{x}$ that achieve the minimum
\begin{equation}
\begin{aligned}
%&
%\text{minimize}
\min_{\bm{x}>\bm{0}}
%&&
&&&
\bm{x}^{-}\bm{1}\bm{1}^{T}\bm{x};
\\
%&
%\text{subject to}
\text{s.t.}
&&&
\bm{B}\bm{x}
\leq
\bm{x}.
\end{aligned}
\label{P-minxx11x-Bxleqx}
\end{equation}

The statement below offers a solution to problem \eqref{P-minxx11x-Bxleqx}, obtained by applying Theorem~\ref{T-minxAx-Bxleqx} with substitution $\bm{A}=\bm{1}\bm{1}^{T}$ (see \cite{Krivulin2021Algebraicsolution} for a solution of the problem with an extended system of constraints).
\begin{lemma}
\label{L-minx11x-Bxleqx}
For any matrix $\bm{B}$ with $\mathop\mathrm{Tr}(\bm{B})\leq1$, the minimum value of the objective function in problem \eqref{P-minxx11x-Bxleqx} is equal to $\|\bm{B}^{\ast}\|$, and all regular solutions are given in the parametric form
\begin{equation*}
\bm{x}
=
\bm{G}\bm{u},
\qquad
\bm{G}
=
\bigoplus_{0\leq i+j\leq n-1}
\|\bm{B}^{\ast}\|^{-1}
%\theta^{-1}
\bm{B}^{i}\bm{1}\bm{1}^{T}\bm{B}^{j}
\oplus
\bm{B}^{\ast},
\qquad
\bm{u}\ne\bm{0}.
\end{equation*}
\end{lemma}

The computational complexity of the solution is shown in \cite{Krivulin2021Algebraicsolution} to be $O(n^{4})$.

Let us examine problems \eqref{P-maxxx11x-Bxleqx} and \eqref{P-minxx11x-Bxleqx} under additional conditions to implement an improved technique of evaluating the best and worst differentiating solutions of the pairwise comparison problem. We add to both problems a normalization constraint on the solution vector $\bm{x}$ in the form
\begin{equation*}
\bm{1}^{T}\bm{x}
=
\|\bm{x}\|
=
1.
\end{equation*}

We now concentrate on finding extremal (minimal and maximal) normalized solutions of problems \eqref{P-maxxx11x-Bxleqx} and \eqref{P-minxx11x-Bxleqx}. To solve the problems with normalization constraint, one can refine the results of Lemmas~\ref{L-maxx11x-Bxleqx} and \ref{L-minx11x-Bxleqx} and then find the extremal solutions. Below we offer direct solutions to the problems, which use less specific assumptions and involve less complicated calculations.

\subsection{Maximization Problem}

Suppose that we need to find the minimal solution of problem \eqref{P-maxxx11x-Bxleqx} with the normalization constraint added. The problem takes the following form:
\begin{equation}
\begin{aligned}
%&
%\text{maximize}
\max_{\bm{x}>\bm{0}}
%&&
&&&
\bm{x}^{-}\bm{1};
\\
%&
%\text{subject to}
\text{s.t.}
&&&
\bm{B}\bm{x}
\leq
\bm{x},
\quad
\bm{1}^{T}\bm{x}
=
1.
\end{aligned}
\label{P-maxxx11x-Bxleqx-1xeq1}
\end{equation}

The next statement improves the result of Lemma~\eqref{L-maxx11x-Bxleqx} and extends it to matrices that may have zero entries.
\begin{lemma}
\label{L-maxxx11x-Bxleqx-1xeq1}
For any matrix $\bm{B}$ with $\mathop\mathrm{Tr}(\bm{B})\leq1$, the maximum value of the objective function in problem \eqref{P-maxxx11x-Bxleqx-1xeq1} is equal to $\|\bm{B}^{\ast}
(\bm{B}^{\ast})^{-}\|$, and the minimal solutions are given by the set of normalized columns in the matrix $\bm{B}^{\ast}=(\bm{b}_{j}^{\ast})$ in the form
\begin{equation*}
\bm{x}
=
\bm{b}_{k}^{\ast}\|\bm{b}_{k}^{\ast}\|^{-1},
\qquad
k
=
\arg\max_{1\leq j\leq n}\|\bm{b}_{j}^{\ast}\|
\|(\bm{b}_{j}^{\ast})^{-}\|.
\end{equation*}

If a vector of the set is dominated by the other vectors (the minimal vector) or is the only vector in the set, it uniquely determines the minimal solution of the problem. Otherwise, the minimal solution is not uniquely defined.
\end{lemma}
\begin{proof}
We solve the inequality constraint in problem \eqref{P-maxxx11x-Bxleqx-1xeq1} by using Lemma~\ref{L-Bxleqx} to write $\bm{x}=\bm{B}^{\ast}\bm{u}$ where $\bm{u}\ne\bm{0}$, and then represent the problem in the form
\begin{equation*}
\begin{aligned}
%&
%\text{minimize}
\max_{\bm{u}\ne\bm{0}}
%&&
&&&
(\bm{B}^{\ast}\bm{u})^{-}\bm{1};
\\
%&
%\text{subject to}
\text{s.t.}
&&&
\bm{1}^{T}\bm{B}^{\ast}\bm{u}
=
1.
%\label{P-maxBastu-1Bastueq1}
\end{aligned}
\end{equation*}

First, we observe that the equality constraint always has solutions. It follows from this constraint that $\bm{u}\leq(\bm{1}^{T}\bm{B}^{\ast})^{-}$ where at least for one component of the vector $\bm{u}$, the scalar inequality is satisfied as an equality. 

Suppose that the equality holds for some $k=1,\ldots,n$, which defines the components of $\bm{u}=(u_{j})$ by the conditions
\begin{equation*}
u_{k}
=
(\bm{1}^{T}\bm{b}_{k}^{\ast})^{-1}
=
\|\bm{b}_{k}^{\ast}\|^{-1},
\qquad
u_{j}
\leq
(\bm{1}^{T}\bm{b}_{j}^{\ast})^{-1}
=
\|\bm{b}_{j}^{\ast}\|^{-1},
\quad
j\ne k.
\end{equation*}

Under this assumption, we use properties of idempotent addition to bound the objective function from above as follows:
\begin{equation*}
(\bm{B}^{\ast}\bm{u})^{-}\bm{1}
=
\bigoplus_{i=1}^{n}
(b_{i1}^{\ast}u_{1}\oplus\cdots\oplus b_{in}^{\ast}u_{n})^{-1}
\leq
\bigoplus_{i=1}^{n}(b_{ik}^{\ast}u_{k})^{-1}
=
\|(\bm{b}_{k}^{\ast})^{-}\|
\|\bm{b}_{k}^{\ast}\|.
\end{equation*}

Note that the above inequality becomes an equality if we put $u_{j}=0$ for all $j\ne k$ since in this case, we have $(b_{i1}^{\ast}u_{1}\oplus\cdots\oplus b_{in}^{\ast}u_{n})^{-1}=(b_{ik}^{\ast}u_{k})^{-1}$.

Furthermore, the maximum value of the objective function $(\bm{B}^{\ast}\bm{u})^{-}\bm{1}$ is attained if the index $k$ is chosen as
\begin{equation*}
k
=
\arg\max_{1\leq j\leq n}
\|\bm{b}_{j}^{\ast}\|\|(\bm{b}_{j}^{\ast})^{-}\|,
\end{equation*}
which yields this maximum equal to $\bm{1}^{T}\bm{B}^{\ast}(\bm{B}^{\ast})^{-}\bm{1}=\|\bm{B}^{\ast}(\bm{B}^{\ast})^{-}\|$.
 
For each index $k$ that provides the above maximum, we take as a solution to the maximization problem the vector $\bm{u}$ with the components
\begin{equation*}
u_{k}
=
\|\bm{b}_{k}^{\ast}\|^{-1},
\qquad
u_{j}
=
0,
\quad
j\ne k.
\end{equation*}

Since the equality $\bm{x}=\bm{B}^{\ast}\bm{u}$ holds, this vector $\bm{u}$ produces the vector
\begin{equation*}
\bm{x}
=
u_{1}\bm{b}_{1}^{\ast}
\oplus\cdots\oplus
u_{n}\bm{b}_{n}^{\ast}
=
\bm{b}_{k}^{\ast}\|\bm{b}_{k}^{\ast}\|^{-1}.
\end{equation*}

We observe that the solution vector obtained has the form of a normalized column of the matrix $\bm{B}^{\ast}$, and that for a fixed $k$, it is the minimal solution due to monotonicity of the matrix operations.

In the case of several indices $k$ that yield the maximum value of the objective function, we combine the corresponding solution vectors in a set. If there is the minimal vector among them, we take this vector as the unique minimal solution. Otherwise, the minimal solution cannot be uniquely determined, and therefore it is given by multiple vectors.
\end{proof}

\subsection{Minimization Problem}

We turn to the derivation of the maximal solution to the following minimization problem obtained from \eqref{P-minxx11x-Bxleqx} by adding the normalization constraint:
\begin{equation}
\begin{aligned}
%&
%\text{minimize}
\min_{\bm{x}>\bm{0}}
%&&
&&&
\bm{x}^{-}\bm{1};
\\
%&
%\text{subject to}
\text{s.t.}
&&&
\bm{B}\bm{x}
\leq
\bm{x},
\quad
\bm{1}^{T}\bm{x}
=
1.
\end{aligned}
\label{P-minxx11x-Bxleqx-1xeq1}
\end{equation}

A complete solution of the problem is given by the next statement.
\begin{lemma}
\label{L-minxx11x-Bxleqx-1xeq1}
For any matrix $\bm{B}$ with $\mathop\mathrm{Tr}(\bm{B})\leq1$, the minimum value of the objective function in problem \eqref{P-minxx11x-Bxleqx-1xeq1} is equal to $\|\bm{B}^{\ast}\|$, and the maximal solution is given by
\begin{equation*}
\bm{x}
=
(\bm{1}^{T}\bm{B}^{\ast})^{-}.
\end{equation*}
\end{lemma}
\begin{proof}
Similarly as before, we replace the inequality $\bm{B}\bm{x}\leq\bm{x}$ by its solution $\bm{x}=\bm{B}^{\ast}\bm{u}$ with $\bm{u}\ne\bm{0}$. Combining with the constraint $\bm{1}^{T}\bm{x}=1$ yields the equality $\bm{1}^{T}\bm{B}^{\ast}\bm{u}=1$, from which the inequality $\bm{u}\leq(\bm{1}^{T}\bm{B}^{\ast})^{-}$ follows. 

After multiplication of the last inequality by $\bm{B}^{\ast}$ from the left, we arrive at a constraint on $\bm{x}$ in the form
\begin{equation*}
\bm{x}
=
\bm{B}^{\ast}\bm{u}
\leq
\bm{B}^{\ast}(\bm{1}^{T}\bm{B}^{\ast})^{-}.
\end{equation*}

In the same way as in \cite{Krivulin2014Complete}, we verify the equality $\bm{B}^{\ast}(\bm{1}^{T}\bm{B}^{\ast})^{-}=(\bm{1}^{T}\bm{B}^{\ast})^{-}$. Indeed, since $\bm{B}^{\ast}\geq\bm{I}$, the inequality $\bm{B}^{\ast}(\bm{1}^{T}\bm{B}^{\ast})^{-}\geq(\bm{1}^{T}\bm{B}^{\ast})^{-}$ holds.

To show the opposite inequality $\bm{B}^{\ast}(\bm{1}^{T}\bm{B}^{\ast})^{-}\leq(\bm{1}^{T}\bm{B}^{\ast})^{-}$, we note that the identity $\bm{B}^{\ast}\bm{B}^{\ast}=\bm{B}^{\ast}$ is valid according to the properties of the Kleene star matrix. Application of the properties of conjugate transposition yields
\begin{equation*}
(\bm{1}^{T}\bm{B}^{\ast})^{-}\bm{1}^{T}\bm{B}^{\ast}
\geq
\bm{I},
\qquad
\bm{1}^{T}\bm{B}^{\ast}\bm{B}^{\ast}(\bm{1}^{T}\bm{B}^{\ast}\bm{B}^{\ast})^{-}=1.
\end{equation*}

We combine these relations to obtain the desired inequality as follows:
\begin{equation*}
\bm{B}^{\ast}
(\bm{1}^{T}\bm{B}^{\ast})^{-}
=
\bm{B}^{\ast}
(\bm{1}^{T}\bm{B}^{\ast}\bm{B}^{\ast})^{-}
\leq
(\bm{1}^{T}\bm{B}^{\ast})^{-}
\bm{1}^{T}\bm{B}^{\ast}
\bm{B}^{\ast}
(\bm{1}^{T}\bm{B}^{\ast}\bm{B}^{\ast})^{-}
=
(\bm{1}^{T}\bm{B}^{\ast})^{-}.
\end{equation*}

As a result, the inequality constraint on the vector $\bm{x}$ reduces to 
\begin{equation*}
\bm{x}
\leq
(\bm{1}^{T}\bm{B}^{\ast})^{-}.
\end{equation*}

Furthermore, we examine the vector $\bm{x}=(\bm{1}^{T}\bm{B}^{\ast})^{-}$, which is the maximal solution of the last inequality, to verify that this vector satisfies both constraints of the problem.
Since $\bm{B}^{\ast}\geq\bm{B}$, we can write
\begin{equation*}
\bm{B}\bm{x}
=
\bm{B}(\bm{1}^{T}\bm{B}^{\ast})^{-}
\leq
\bm{B}^{\ast}(\bm{1}^{T}\bm{B}^{\ast})^{-}
=
(\bm{1}^{T}\bm{B}^{\ast})^{-}
=
\bm{x},
\end{equation*}
and thus this vector $\bm{x}$ satisfies the inequality constraint. Moreover, we have
\begin{equation*}
\bm{1}^{T}\bm{x}
=
\bm{1}^{T}(\bm{1}^{T}\bm{B}^{\ast})^{-}
=
\bm{1}^{T}\bm{B}^{\ast}(\bm{1}^{T}\bm{B}^{\ast})^{-}
=
1,
\end{equation*}
which means that the equality constraint holds as well.

After conjugate transposition of both sides of the inequality $\bm{x}\leq(\bm{1}^{T}\bm{B}^{\ast})^{-}$ and multiplication by $\bm{1}$, we obtain a lower bound on the objective function
\begin{equation*}
\bm{x}^{-}\bm{1}
\geq
\bm{1}^{T}\bm{B}^{\ast}\bm{1}
=
\|\bm{B}^{\ast}\|.
\end{equation*}

It is clear that the maximal feasible solution $\bm{x}=(\bm{1}^{T}\bm{B}^{\ast})^{-}$ attains this bound, which means that $\|\bm{B}^{\ast}\|$ is the minimum of the objective function. 

Note that the solution obtained is unique and can be represented as
\begin{equation*}
\bm{x}
=
\begin{pmatrix}
\|\bm{b}_{1}^{\ast}\|^{-1}
&
\ldots
&
\|\bm{b}_{n}^{\ast}\|^{-1}
\end{pmatrix}^{T},
\end{equation*}
which is the vector of inverses of the maximum column elements in $\bm{B}^{\ast}$.
\end{proof}

Let us verify that the solution $(\bm{1}^{T}\bm{B}^{\ast})^{-}$ of the minimization problem \eqref{P-minxx11x-Bxleqx-1xeq1} is greater or equal to any solution of the maximization problem \eqref{P-maxxx11x-Bxleqx-1xeq1}. Indeed, with the equality $(\bm{1}^{T}\bm{B}^{\ast})^{-}=\bm{B}^{\ast}(\bm{1}^{T}\bm{B}^{\ast})^{-}$, we have for each $k=1,\ldots,n$ that
\begin{equation*}
(\bm{1}^{T}\bm{B}^{\ast})^{-}
=
\bm{B}^{\ast}(\bm{1}^{T}\bm{B}^{\ast})^{-}
=
\bigoplus_{j=1}^{n}
\bm{b}_{j}^{\ast}\|\bm{b}_{j}^{\ast}\|^{-1}
\geq
\bm{b}_{k}^{\ast}\|\bm{b}_{k}^{\ast}\|^{-1},
\end{equation*}
which shows that all normalized columns of the matrix $\bm{B}^{\ast}$, including the solution vectors of problem \eqref{P-maxxx11x-Bxleqx-1xeq1}, are less or equal to the solution of \eqref{P-minxx11x-Bxleqx-1xeq1}.

To conclude this section, we note that the obtained solutions to both maximization and minimization problems involve the calculation of the Kleene star matrix $\bm{B}^{\ast}$ as the main component, and thus have the complexity $O(n^{4})$.

\section{Solution of Multicriteria Pairwise Comparison Problems}
\label{S-SMPCP}

Consider the multiobjective optimization problem \eqref{P-minx_maxijc1ijxixj_maxijcmijxixj-maxbijxjleqxi}, which arises in the solution of the constrained multicriteria pairwise comparison problem using log-Chebyshev approximation. After rewriting the objective functions and inequality constraints in the max-algebra setting, the problem becomes
\begin{equation*}
\begin{aligned}
%&
%\text{minimize}
\min_{\bm{x}>\bm{0}}
%&&
&&&
\left(
\bigoplus_{1\leq i,j\leq n}x_{i}^{-1}c_{ij}^{(1)}x_{j},
\ldots,
\bigoplus_{1\leq i,j\leq n}x_{i}^{-1}c_{ij}^{(m)}x_{j}
\right);
\\
%&
%\text{subject to}
\text{s.t.}
%&&
&&&
\bigoplus_{1\leq j\leq n}b_{ij}x_{j}
\leq
x_{i},
\qquad
i=1,\ldots,n.
\end{aligned}
%\label{P-minx_maxijc1ijxixj_maxijcmijxixj-maxbijxjleqxi}
\end{equation*}

By using vector and matrix notation, we finally formulate the problem as follows. Given $(n\times n)$-matrices $\bm{C}_{l}$ of pairwise comparisons of $n$ alternatives for criteria $l=1,\ldots,m$, and nonnegative $(n\times n)$-matrix $\bm{B}$ of constraints, find positive $n$-vectors $\bm{x}$ of ratings that solve the multiobjective problem
\begin{equation}
\begin{aligned}
%&
%\text{minimize}
\min_{\bm{x}>\bm{0}}
%&&
&&&
(\bm{x}^{-}\bm{C}_{1}\bm{x},\ldots,\bm{x}^{-}\bm{C}_{m}\bm{x});
\\
%&
%\text{subject to}
\text{s.t.}
%&&
&&&
\bm{B}\bm{x}
\leq
\bm{x}.
\label{P-minx_xC1x_xCmx-Bxleqx}
\end{aligned}
\end{equation}

In this section, we offer new max-ordering, lexicographic and lexicographic max-ordering optimal solutions to the problem. If the solution obtained is not a unique vector of ratings, this solution is reduced to two the best and worst solution vectors, which most and least differentiate between the alternatives with the highest and lowest ratings.

\subsection{Max-Ordering Solution}

We start with the max-ordering solution technique, which leads to the solution of problem \eqref{P-minx_maxijaijxixj} over the feasible set $X_{0}$ given by \eqref{E-S0}.

In terms of max-algebra, the set $X_{0}$ is the set of regular vectors $\bm{x}$ that solve the inequality $\bm{B}\bm{x}\leq\bm{x}$. The objective function at \eqref{P-minx_maxijaijxixj} takes the form
\begin{equation*}
\max_{1\leq l\leq m}
\bm{x}^{-}\bm{C}_{l}\bm{x}
=
\bm{x}^{-}\bm{A}\bm{x},
\qquad
\bm{A}
=
%\bm{C}_{1}\oplus\cdots\oplus\bm{C}_{m}.
\bigoplus_{l=1}^{m}
\bm{C}_{l}.
\end{equation*}

Combining the objective function with the constraint leads to the problem
\begin{equation*}
\begin{aligned}
%&
%\text{minimize}
\min_{\bm{x}>\bm{0}}
%&&
&&&
\bm{x}^{-}\bm{A}\bm{x};
\\
%&
%\text{subject to}
\text{s.t.}
%&&
&&&
\bm{B}\bm{x}
\leq
\bm{x}.
%\label{P-minxopluskxAkx-Bxleqx}
\end{aligned}
\end{equation*}

Since the problem takes the form of \eqref{P-minxAx-Bxleqx}, we apply Theorem~\ref{T-minxAx-Bxleqx} to find the minimum value $\theta$ of the objective function, which is given by \eqref{E-thetaeqtr1nABi1ABik}. Next, we obtain the solution set $X_{1}$ of the problem as the set of vectors
\begin{equation*}
\bm{x}
=
%(\theta^{-1}\bm{A}\oplus\bm{B})^{\ast}\bm{u},
\bm{B}_{1}^{\ast}\bm{u},
\qquad
\bm{B}_{1}
=
\theta^{-1}\bm{A}\oplus\bm{B},
\qquad
\bm{u}\ne\bm{0}.
\end{equation*}

If the columns in the matrix $\bm{G}=\bm{B}_{1}^{\ast}$ are collinear, then any column after normalization can be taken as the unique solution. Otherwise by Theorem~\ref{T-minxAx-Bxleqx-equivalence}, the set $X_{1}$ coincides with the set of regular solutions to the inequality
\begin{equation*}
\bm{B}_{1}\bm{x}
\leq
\bm{x}.
\end{equation*}

It remains to obtain the best and worst differentiating solutions in the set, which are the minimal and maximal solutions of the respective problems 
\begin{equation*}
\begin{aligned}
%&
%\text{maximize}
\max_{\bm{x}>\bm{0}}
%&&
&&&
\bm{x}^{-}\bm{1};
\\
%&
%\text{subject to}
\text{s.t.}
&&&
\bm{B}_{1}\bm{x}
\leq
\bm{x},
\quad
%\\
%&&&
\bm{1}^{T}\bm{x}
=
1;
\end{aligned}
\qquad\qquad
\begin{aligned}
%&
%\text{minimize}
\min_{\bm{x}>\bm{0}}
%&&
&&&
\bm{x}^{-}\bm{1};
\\
%&
%\text{subject to}
\text{s.t.}
&&&
\bm{B}_{1}\bm{x}
\leq
\bm{x},
\quad
%\\
%&&&
\bm{1}^{T}\bm{x}
=
1.
\end{aligned}
\end{equation*}

By applying Lemma~\ref{L-maxxx11x-Bxleqx-1xeq1}, we find the best differentiating solutions represented in terms of the columns of the matrix $\bm{G}=(\bm{g}_{j})$ as the set of vectors
\begin{equation*}
\bm{x}^{\textup{best}}
=
\bm{g}_{k}\|\bm{g}_{k}\|^{-1},
\qquad
k
=
\arg\max_{1\leq j\leq n}\|\bm{g}_{j}\|
\|(\bm{g}_{j})^{-}\|.
\end{equation*}

If there is a minimal solution in the set, it is taken as the unique minimal best differentiating solution; otherwise, the minimal solution cannot be uniquely determined.

By Lemma~\ref{L-minxx11x-Bxleqx-1xeq1}, the maximal worst differentiating solution is obtained as
\begin{equation*}
\bm{x}^{\textup{worst}}
=
(\bm{1}^{T}\bm{G})^{-}.
\end{equation*}
 
The next statement combines the results of the derivation of all max-ordering solution vectors for problem \eqref{P-minx_xC1x_xCmx-Bxleqx} with the calculation of the best and worst differentiating vectors among them.  
\begin{theorem}
\label{T-minx_xC1x_xCmx-Bxleqx_MO}
Let $\bm{C}_{l}$ for all $l=1,\ldots,m$ be matrices with nonzero spectral radii and $\bm{B}$ be a matrix such that $\mathop\mathrm{Tr}(\bm{B})\leq1$. Define
\begin{gather*}
\theta
=
\bigoplus_{k=1}^{n}
\bigoplus_{0\leq i_{1}+\cdots+i_{k}\leq n-k}
\mathop\mathrm{tr}\nolimits^{1/k}(\bm{A}\bm{B}^{i_{1}}\cdots\bm{A}\bm{B}^{i_{k}}),
\qquad
\bm{A}
=
\bigoplus_{l=1}^{m}
\bm{C}_{l},
\\
\bm{B}_{1}
=
\theta^{-1}\bm{A}
\oplus
\bm{B}.
\end{gather*}
Then, with the notation $\bm{G}=\bm{B}_{1}^{\ast}$, the following statements hold:
\renewcommand{\labelenumi}{(\theenumi)}%
\renewcommand{\theenumi}{\roman{enumi}}%
\begin{enumerate}
%\item
%The minimum value of the objective function in problem \eqref{P-minx_xC1x_xCmx-Bxleqx} is equal to $\theta$.
\item
All max-ordering solutions of problem \eqref{P-minx_xC1x_xCmx-Bxleqx} are given by the matrix $\bm{G}=(\bm{g}_{j})$ in the parametric form
\begin{equation*}
\bm{x}
=
\bm{G}\bm{u},
\qquad
\bm{u}\ne\bm{0}.
%\label{E-xeqtheta1ABastu}
\end{equation*}
\item
The minimal normalized best differentiating solution is given by
%The minimal normalized best differentiating solution is given by the minimal vectors in the set of vectors
\begin{equation*}
\bm{x}^{\textup{best}}
=
\bm{g}_{k}\|\bm{g}_{k}\|^{-1},
\qquad
k
=
\arg\max_{1\leq j\leq n}
\|\bm{g}_{j}\|\|\bm{g}_{j}^{-}\|.
\end{equation*}
\item
The maximal normalized worst differentiating solution is
\begin{equation*}
\bm{x}^{\textup{worst}}
=
(\bm{1}^{T}\bm{G})^{-}.
\end{equation*}
\end{enumerate}
\end{theorem}

We observe that if the matrix $\bm{G}$ produces a unique solution, then the minimal best and maximal worst differentiating solutions obviously coincide. This allows one not to check whether all columns in $\bm{G}$ are collinear or not. Instead one can directly calculate the best and worst solutions, which provide a common result if the solution is unique. 

The above result is valid for the unconstrained problem which is obtained by setting $\bm{B}=\bm{0}$. In this case, the minimum value $\theta$ of the objective function is reduced to the spectral radius of the matrix $\bm{A}$.

Finally, note that with additional computational effort required to find the matrix $\bm{A}$, the computational complexity of the solution is $O(mn^{2}+n^{5})$.

\subsection{Lexicographic Ordering Solution}

Consider the implementation of the lexicographic ordering technique which involves a series of problems \eqref{P-minx_maxijcijsxixj}. We solve problem \eqref{P-minx_xC1x_xCmx-Bxleqx} in no more than $m$ steps each consisting in the minimization of a scalar objective function over a feasible set given by the solutions of the previous step.

At the step $s=1$, we use the symbol $\bm{B}_{0}=\bm{B}$ and formulate the problem
\begin{equation*}
\begin{aligned}
%&
%\text{minimize}
\min_{\bm{x}>\bm{0}}
%&&
&&&
\bm{x}^{-}\bm{C}_{1}\bm{x};
\\
%&
%\text{subject to}
\text{s.t.}
&&&
\bm{B}_{0}\bm{x}
\leq
\bm{x}.
\end{aligned}
\end{equation*}

To solve the problem, we apply Theorem~\ref{T-minxAx-Bxleqx} to calculate the minimum
\begin{equation*}
\theta_{1}
=
\bigoplus_{k=1}^{n}
\bigoplus_{0\leq i_{1}+\cdots+i_{k}\leq n-k}
\mathop\mathrm{tr}\nolimits^{1/k}(\bm{C}_{1}\bm{B}_{0}^{i_{1}}\cdots\bm{C}_{1}\bm{B}_{0}^{i_{k}}).
\end{equation*}

Then, we find the solution set, which is given in parametric form by 
\begin{equation*}
\bm{x}
=
\bm{B}_{1}^{\ast}\bm{u},
\qquad
\bm{B}_{1}
=
\theta_{1}^{-1}\bm{C}_{1}
\oplus
\bm{B}_{0},
\qquad
\bm{u}
\ne
\bm{0},
\end{equation*}
or according to Theorem~\ref{T-minxAx-Bxleqx-equivalence}, as the set of solutions of the inequality 
\begin{equation*}
\bm{B}_{1}\bm{x}
%=
%(\theta_{1}^{-1}\bm{A}_{1}\oplus\bm{B}_{0})\bm{x}
\leq
\bm{x}.
\end{equation*}

We take the last inequality as a constraint that determines the set $X_{1}$ and formulate the problem of step $s=2$ as follows:
\begin{equation*}
\begin{aligned}
%&
%\text{minimize}
\min_{\bm{x}>\bm{0}}
%&&
&&&
\bm{x}^{-}\bm{C}_{2}\bm{x};
\\
%&
%\text{subject to}
\text{s.t.}
&&&
\bm{B}_{1}\bm{x}
\leq
\bm{x}.
\end{aligned}
\end{equation*}

The minimum in the problem is given by
\begin{equation*}
\theta_{2}
=
\bigoplus_{k=1}^{n}
\bigoplus_{0\leq i_{1}+\cdots+i_{k}\leq n-k}
\mathop\mathrm{tr}\nolimits^{1/k}(\bm{C}_{2}\bm{B}_{1}^{i_{1}}\cdots\bm{C}_{2}\bm{B}_{1}^{i_{k}}),
\end{equation*}
and the set of solutions $X_{2}$ is defined as
\begin{equation*}
\bm{x}
=
\bm{B}_{2}^{\ast}\bm{u},
\qquad
\bm{B}_{2}
=
\theta_{2}^{-1}\bm{C}_{2}
\oplus
\bm{B}_{1},
\qquad
\bm{u}
\ne
\bm{0}.
\end{equation*}

We repeat the procedure for each step $s=3,\ldots,m$. Upon completion of step $s=m$, we arrive at the lexicographic solution given by
\begin{equation*}
\bm{x}
=
\bm{B}_{m}^{\ast}\bm{u},
\qquad
\bm{u}
\ne
\bm{0}.
\end{equation*}

If the solution obtained is not unique (up to a positive factor), we calculate the best and worst differentiating solution vectors. 

We summarize the above computational scheme in the following form.
\begin{theorem}
\label{T-minx_xC1x_xCmx-Bxleqx_LO}
Let $\bm{C}_{l}$ for all $l=1,\ldots,m$ be matrices with nonzero spectral radii and $\bm{B}$ be a matrix such that $\mathop\mathrm{Tr}(\bm{B})\leq1$. Denote $\bm{B}_{0}=\bm{B}$ and define the recurrence relations
\begin{gather*}
\theta_{s}
=
\bigoplus_{k=1}^{n}
\bigoplus_{0\leq i_{1}+\cdots+i_{k}\leq n-k}
\mathop\mathrm{tr}\nolimits^{1/k}(\bm{C}_{s}\bm{B}_{s-1}^{i_{1}}\cdots\bm{C}_{s}\bm{B}_{s-1}^{i_{k}}),
\qquad
\\
\bm{B}_{s}
=
\theta_{s}^{-1}\bm{C}_{s}
\oplus
\bm{B}_{s-1},
\qquad
s=1,\ldots,m.
\end{gather*}
Then, with the notation $\bm{G}=\bm{B}_{m}^{\ast}$, the following statements hold:
\renewcommand{\labelenumi}{(\theenumi)}%
\renewcommand{\theenumi}{\roman{enumi}}%
\begin{enumerate}
%\item
%The minimum value of the objective function in problem \eqref{P-minx_xC1x_xCmx-Bxleqx} is equal to $\theta$.
\item
All lexicographic ordering solutions of problem \eqref{P-minx_xC1x_xCmx-Bxleqx} are given by the matrix $\bm{G}=(\bm{g}_{j})$ in the parametric form
\begin{equation*}
\bm{x}
=
\bm{G}\bm{u},
\qquad
\bm{u}\ne\bm{0}.
%\label{E-xeqtheta1ABastu}
\end{equation*}
\item
The minimal normalized best differentiating solution is given by
\begin{equation*}
\bm{x}^{\textup{best}}
=
\bm{g}_{k}\|\bm{g}_{k}\|^{-1},
\qquad
k
=
\arg\max_{1\leq j\leq n}
\|\bm{g}_{j}\|\|\bm{g}_{j}^{-}\|.
\end{equation*}
\item
The maximal normalized worst differentiating solution is
\begin{equation*}
\bm{x}^{\textup{worst}}
=
(\bm{1}^{T}\bm{G})^{-}.
\end{equation*}
\end{enumerate}
\end{theorem}

Suppose that for some step $s<m$, all columns in the matrix $\bm{B}_{s}^{\ast}$ are collinear, and thus the matrix generates a unique solution vector. In this case, further steps cannot change the obtained solution and thus can be avoided to stop the procedure. Note that if a matrix $\bm{B}_{s}^{\ast}$ generates a unique solution, then both the minimal best and maximal worst normalized vectors obtained from $\bm{B}_{s}^{\ast}$ are the same. As a result, we can calculate these vectors to stop the procedure if these vectors coincide or continue otherwise.

The complexity of the solution can be estimated not greater than $O(mn^{5})$.

\subsection{Lexicographic Max-Ordering Solution}

We now describe a solution technique based on the lexicographic max-ordering optimality principle. Similar to the lexicographic ordering solution, we handle problem \eqref{P-minx_xC1x_xCmx-Bxleqx} by solving a series of problems, where each problem has a scalar objective function and inequality constraint provided by the solution of the previous problem. According to the computational scheme given by formulas \eqref{E-S0}, \eqref{P-minx_maxkl1maxijcijlxixj} and \eqref{E-Is}, the solution involves no more than $m$ steps, which are described in the framework of max-algebra as follows.

We set $\bm{B}_{0}=\bm{B}$ and define $X_{0}=\{\bm{x}>\bm{0}|\ \bm{B}_{0}\bm{x}\leq\bm{x}\}$, $I_{0}=\{1,\ldots,m\}$.

The step $s=1$ starts with calculating the matrix
\begin{equation*}
\bm{A}_{1}
=
\bigoplus_{l\in I_{0}}\bm{C}_{l}
=
\bm{C}_{1}\oplus\cdots\oplus\bm{C}_{m}.
\end{equation*}

The purpose of this step is to solve the problem
\begin{equation*}
\begin{aligned}
%&
%\text{minimize}
\min_{\bm{x}>\bm{0}}
%&&
&&&
\bm{x}^{-}\bm{A}_{1}\bm{x};
\\
%&
%\text{subject to}
\text{s.t.}
&&&
\bm{B}_{0}\bm{x}
\leq
\bm{x}.
\end{aligned}
\end{equation*}

We apply Theorem~\ref{T-minxAx-Bxleqx} to find the minimum of the objective function
\begin{equation*}
\theta_{1}
=
\bigoplus_{k=1}^{n}
\bigoplus_{0\leq i_{1}+\cdots+i_{k}\leq n-k}
\mathop\mathrm{tr}\nolimits^{1/k}(\bm{A}_{1}\bm{B}_{0}^{i_{1}}\cdots\bm{A}_{1}\bm{B}_{0}^{i_{k}}),
\end{equation*}
and then obtain all solutions in the parametric form
\begin{equation*}
\bm{x}
=
\bm{B}_{1}^{\ast}\bm{u},
\qquad
\bm{B}_{1}
=
\theta_{1}^{-1}\bm{A}_{1}
\oplus
\bm{B}_{0},
\qquad
\bm{u}
\ne
\bm{0}.
\end{equation*}
 
By Theorem~\ref{T-minxAx-Bxleqx-equivalence} these solutions can be defined by the inequality $\bm{B}_{1}\bm{x}\leq\bm{x}$ to provide the new feasible set 
\begin{equation*}
X_{1}
=
\{\bm{x}>\bm{0}:\ \bm{B}_{1}\bm{x}\leq\bm{x}\}.
\end{equation*}

To prepare the next step, we find the minimums
\begin{equation*}
\theta_{1l}
=
\min_{\bm{x}\in X_{1}}
\bm{x}^{-}\bm{C}_{l}\bm{x}
=
\bigoplus_{k=1}^{n}
\bigoplus_{0\leq i_{1}+\cdots+i_{k}\leq n-k}
\mathop\mathrm{tr}\nolimits^{1/n}(\bm{C}_{l}\bm{B}_{1}^{i_{1}}\cdots\bm{C}_{l}\bm{B}_{1}^{i_{k}}),
\qquad
l\in I_{0}.
\end{equation*}

At the step $s=2$, we form the matrix
\begin{equation*}
\bm{A}_{2}
=
\bigoplus_{l\in I_{1}}\bm{C}_{l},
\qquad
I_{1}
=
\left\{
l\in I_{0}\
:\
\theta_{1}
>
\theta_{1l}
\right\},
\end{equation*}
and then solve the problem
\begin{equation*}
\begin{aligned}
%&
%\text{minimize}
\min_{\bm{x}>\bm{0}}
%&&
&&&
\bm{x}^{-}\bm{A}_{2}\bm{x};
\\
%&
%\text{subject to}
\text{s.t.}
&&&
\bm{B}_{1}\bm{x}
\leq
\bm{x}.
\end{aligned}
\end{equation*}

The minimum in the problem is given by
\begin{equation*}
\theta_{2}
=
\bigoplus_{k=1}^{n}
\bigoplus_{0\leq i_{1}+\cdots+i_{k}\leq n-k}
\mathop\mathrm{tr}\nolimits^{1/k}(\bm{A}_{2}\bm{B}_{1}^{i_{1}}\cdots\bm{A}_{2}\bm{B}_{1}^{i_{k}}),
\end{equation*}
and the set of solutions is
\begin{equation*}
X_{2}
=
\{\bm{x}>\bm{0}:\ \bm{B}_{2}\bm{x}\leq\bm{x}\},
\qquad
\bm{B}_{2}
=
\theta_{2}^{-1}\bm{A}_{2}
\oplus
\bm{B}_{1}.
\end{equation*}

To complete this step, we calculate the values
\begin{equation*}
\theta_{2l}
=
\bigoplus_{k=1}^{n}
\bigoplus_{0\leq i_{1}+\cdots+i_{k}\leq n-k}
\mathop\mathrm{tr}\nolimits^{1/n}(\bm{C}_{l}\bm{B}_{2}^{i_{1}}\cdots\bm{C}_{l}\bm{B}_{2}^{i_{k}}),
\qquad
l\in I_{1},
\end{equation*}
and then define
\begin{equation*}
\bm{A}_{3}
=
\bigoplus_{l\in I_{2}}\bm{C}_{l},
\qquad
I_{2}
=
\left\{
l\in I_{1}\
:\
\theta_{2}
>
\theta_{2l}
\right\}.
\end{equation*}

We repeat the procedure for all remaining steps $s\leq m$. The procedure stops at step $s$ if the set $X_{s}$ consists of a single solution, or $I_{s}=\emptyset$. 

The above solution scheme can be summarized as follows.
\begin{theorem}
\label{T-minx_xC1x_xCmx-Bxleqx_LMO}
Let $\bm{C}_{l}$ for all $l=1,\ldots,m$ be matrices with nonzero spectral radii and $\bm{B}$ be a matrix such that $\mathop\mathrm{Tr}(\bm{B})\leq1$. Denote $\bm{B}_{0}=\bm{B}$ and $I_{0}=\{1,\ldots,m\}$, and define the recurrence relations
\begin{gather*}
\theta_{s}
=
\bigoplus_{k=1}^{n}
\bigoplus_{0\leq i_{1}+\cdots+i_{k}\leq n-k}
\mathop\mathrm{tr}\nolimits^{1/k}(\bm{A}_{s}\bm{B}_{s-1}^{i_{1}}\cdots\bm{A}_{s}\bm{B}_{s-1}^{i_{k}}),
\qquad
\bm{A}_{s}
=
\bigoplus_{l\in I_{s-1}}\bm{C}_{l},
\\
\bm{B}_{s}
=
\theta_{s}^{-1}\bm{A}_{s}
\oplus
\bm{B}_{s-1},
\qquad
I_{s}
=
\left\{
l\in I_{s-1}\
:\
\theta_{s}
>
\theta_{sl}
\right\},
\\
\theta_{sl}
=
\bigoplus_{k=1}^{n}
\bigoplus_{0\leq i_{1}+\cdots+i_{k}\leq n-k}
\mathop\mathrm{tr}\nolimits^{1/n}(\bm{C}_{l}\bm{B}_{s}^{i_{1}}\cdots\bm{C}_{l}\bm{B}_{s}^{i_{k}}),
\quad
l\in I_{s-1},
\qquad
s=1,\ldots,m.
\end{gather*}
Then, with the notation $\bm{G}=\bm{B}_{m}^{\ast}$, the following statements hold:
\renewcommand{\labelenumi}{(\theenumi)}%
\renewcommand{\theenumi}{\roman{enumi}}%
\begin{enumerate}
%\item
%The minimum value of the objective function in problem \eqref{P-minx_xC1x_xCmx-Bxleqx} is equal to $\theta$.
\item
All lexicographic max-ordering solutions of problem \eqref{P-minx_xC1x_xCmx-Bxleqx} are given by the matrix $\bm{G}=(\bm{g}_{j})$ in the parametric form
\begin{equation*}
\bm{x}
=
\bm{G}\bm{u},
\qquad
\bm{u}\ne\bm{0}.
%\label{E-xeqtheta1ABastu}
\end{equation*}
\item
The minimal normalized best differentiating solution is given by
\begin{equation*}
\bm{x}^{\textup{best}}
=
\bm{g}_{k}\|\bm{g}_{k}\|^{-1},
\qquad
k
=
\arg\max_{1\leq j\leq n}
\|\bm{g}_{j}\|\|\bm{g}_{j}^{-}\|.
\end{equation*}
\item
The maximal normalized worst differentiating solution is
\begin{equation*}
\bm{x}^{\textup{worst}}
=
(\bm{1}^{T}\bm{G})^{-}.
\end{equation*}
\end{enumerate}
\end{theorem}

Considering evaluation of the minimums $\theta_{sl}$ at each step $s$, we see that the computational complexity of the solution is not more than $O(m^{2}n^{5})$.

\section{Illustrative Examples}
\label{S-IE}

In this section, we present numerical examples intended to illustrate the solution technique developed. We consider a constrained multicriteria problem of pairwise comparisons, which is formulated as a multiobjective optimization problem of constrained log-Chebyshev matrix approximation, and then solved in the framework of tropical optimization.

Suppose that there are $n=4$ alternatives that are compared in pairs according to $m=4$ criteria. The results of comparisons are given by the following pairwise comparison matrices:
\begin{gather*}
\bm{C}_{1}
=
\left(
\begin{array}{cccc}
  1 &   2 & 3 &   4 \\
1/2 &   1 & 3 &   2 \\
1/3 & 1/3 & 1 & 1/3 \\
1/4 & 1/2 & 3 &   1
\end{array}
\right),
\qquad
\bm{C}_{2}
=
\left(
\begin{array}{cccc}
  1 &   2 &   3 & 4 \\
1/2 &   1 &   2 & 3 \\
1/3 & 1/2 &   1 & 2 \\
1/4 & 1/3 & 1/2 & 1
\end{array}
\right),
\\
\bm{C}_{3}
=
\left(
\begin{array}{cccc}
  1 &   3 & 2 & 3 \\
1/3 &   1 & 2 & 4 \\
1/2 & 1/2 & 1 & 1 \\
1/3 & 1/4 & 1 & 1
\end{array}
\right),
\qquad
\bm{C}_{4}
=
\left(
\begin{array}{cccc}
  1 &   2 &   2 & 1 \\
1/2 &   1 & 1/2 & 3 \\
1/2 &   2 &   1 & 2 \\
  1 & 1/3 & 1/2 & 1
\end{array}
\right).
\end{gather*}

The problem is to evaluate the vector $\bm{x}=(x_{1},x_{2},x_{3},x_{4})^{T}$ of individual ratings of alternatives subject to the constraint $x_{3}\geq x_{4}$, which specifies that the rating of  alternative $3$ cannot be less than the rating of alternative $4$.  

In terms of max-algebra, the problem takes the form of \eqref{P-minx_xC1x_xCmx-Bxleqx} where the constraint matrix is given by
\begin{equation*}
\bm{B}
=
\left(
\begin{array}{cccc}
0 & 0 & 0 & 0 \\
0 & 0 & 0 & 0 \\
0 & 0 & 0 & 1 \\
0 & 0 & 0 & 0
\end{array}
\right).
\end{equation*}

Note that a straightforward analysis of the pairwise comparison matrices shows that according to their entries, alternatives $1$ and $2$ should respectively receive the first and second highest ratings. Alternatives $3$ and $4$ have lower ratings and must satisfy the constraint specified in the problem. 

Below we describe solutions obtained according to the max-ordering, lexicographic ordering and lexicographic max-ordering principles of optimality.

\subsection{Max-Ordering Solution}

To obtain the max-ordering solution, we apply Theorem~\ref{T-minx_xC1x_xCmx-Bxleqx_MO}. Under the same notation as in this theorem, we calculate
\begin{equation*}
\bm{A}
=
\left(
\begin{array}{cccc}
  1 &   3 & 3 & 4 \\
1/2 &   1 & 3 & 4 \\
1/2 &   2 & 1 & 2 \\
  1 & 1/2 & 3 & 1
\end{array}
\right),
\qquad
\theta
=
3.
\end{equation*}
 
Furthermore, we form the matrices
\begin{equation*} 
\bm{B}_{1}
=
\left(
\begin{array}{cccc}
1/3 &   1 &   1 & 4/3 \\
1/6 & 1/3 &   1 & 4/3 \\
1/6 & 2/3 & 1/3 &   1 \\
1/3 & 1/6 &   1 & 1/3
\end{array}
\right),
\qquad
\bm{G}
=
\bm{B}_{1}^{\ast}
=
\left(
\begin{array}{cccc}
  1 &   1 & 4/3 & 4/3 \\
4/9 &   1 & 4/3 & 4/3 \\
1/3 & 2/3 &   1 &   1 \\
1/3 & 2/3 &   1 &   1
\end{array}
\right).
\end{equation*}

%G_equiv =
%[  1,   1, 4/3]
%[4/9,   1, 4/3]
%[1/3, 2/3,   1]
%[1/3, 2/3,   1]

Evaluation of the minimal normalized best differentiating and maximal normalized worst differentiating solutions from the matrix $\bm{G}$ gives
\begin{equation*} 
\bm{x}^{\textup{best}}
=
\left(
\begin{array}{c}
  1 \\ %1.000000000000000
4/9 \\ %0.444444444444444
1/3 \\ %0.333333333333333
1/3    %0.333333333333333
\end{array}
\right)
\approx
\left(
\begin{array}{c}
1.0000 \\ %1.000000000000000
0.4444 \\ %0.444444444444444
0.3333 \\ %0.333333333333333
0.3333    %0.333333333333333
\end{array}
\right),
\qquad 
\bm{x}^{\textup{worst}}
=
\left(
\begin{array}{c}
  1 \\ %1.000000000000000
  1 \\ %1.000000000000000
3/4 \\ %0.750000000000000
3/4    %0.750000000000000
\end{array}
\right)
=
\left(
\begin{array}{c}
1.0000 \\ %1.000000000000000
1.0000 \\ %1.000000000000000
0.7500 \\ %0.750000000000000
0.7500    %0.750000000000000
\end{array}
\right).
\end{equation*} 

We can combine both solutions in a vector where some entries are given in interval form as
\begin{equation*} 
\bm{x}
\approx
\left(
\begin{array}{c}
                1.0000 \\ %1.000000000000000-1.000000000000000
%0.4444\ \textup{--}\ 1.0000 \\ %0.444444444444444-1.000000000000000
%0.3333\ \textup{--}\ 0.7500 \\ %0.333333333333333-0.750000000000000
%0.3333\ \textup{--}\ 0.7500    %0.333333333333333-0.750000000000000
0.4444\ \ldots\ 1.0000 \\ %0.444444444444444-1.000000000000000
0.3333\ \ldots\ 0.7500 \\ %0.333333333333333-0.750000000000000
0.3333\ \ldots\ 0.7500    %0.333333333333333-0.750000000000000
\end{array}
\right).
\end{equation*}

\subsection{Lexicographic Ordering Solution}

According to Theorem~\ref{T-minx_xC1x_xCmx-Bxleqx_LO}, the lexicographic ordering solution consists of several steps and proceeds as follows. First, we calculate  
\begin{equation*}
\theta_{1}
=
3,
\end{equation*}
and then obtain the matrices
\begin{equation*}
\bm{B}_{1}
=
\left(
\begin{array}{cccc}
 1/3 & 2/3 &   1 & 4/3 \\
 1/6 & 1/3 &   1 & 2/3 \\
 1/9 & 1/9 & 1/3 &   1 \\
1/12 & 1/6 &   1 & 1/3
\end{array}
\right),
\qquad
\bm{B}_{1}^{\ast}
=
\left(
\begin{array}{cccc}
  1 & 2/3 & 4/3 & 4/3 \\
1/6 &   1 &   1 &   1 \\
1/9 & 1/6 &   1 &   1 \\
1/9 & 1/6 &   1 &   1
\end{array}
\right).
\end{equation*}

%G1_equiv =
%  1 & 2/3 & 4/3\\
%1/6 &   1 &   1\\
%1/9 & 1/6 &   1\\
%1/9 & 1/6 &   1\\

To check whether the matrix $B_{1}^{\ast}$ generates a nonunique solution, we find the corresponding best and worst differentiating solution vectors 
\begin{equation*}
\bm{x}_{1}^{\textup{best}}
=
\left(
\begin{array}{c}
  1 \\ %1.000000000000000
1/6 \\ %0.166666666666667
1/9 \\ %0.111111111111111
1/9    %0.111111111111111
\end{array}
\right),
\qquad 
\bm{x}_{1}^{\textup{worst}}
=
\left(
\begin{array}{c}
  1 \\ %1.000000000000000
  1 \\ %1.000000000000000
3/4 \\ %0.750000000000000
3/4    %0.750000000000000
\end{array}
\right).
\end{equation*} 

Since these vectors are different, we continue to the next step to evaluate
\begin{equation*}
\theta_{2}
=
2.
\end{equation*} 
 
Further calculations lead to the matrices
\begin{equation*}
\bm{B}_{2}
=
\left(
\begin{array}{cccc}
1/2 &   1 & 3/2 &   2 \\
1/4 & 1/2 &   1 & 3/2 \\
1/6 & 1/4 & 1/2 &   1 \\
1/8 & 1/6 &   1 & 1/2
\end{array}
\right),
\qquad
\bm{B}_{2}^{\ast}
=
\left(
\begin{array}{cccc}
  1 &   1 &   2 &   2 \\
1/4 &   1 & 3/2 & 3/2 \\
1/6 & 1/4 &   1 &   1 \\
1/6 & 1/4 &   1 &   1
\end{array}
\right).
\end{equation*}
 
%G2_equiv =
%  1 &   1 &   2\\
%1/4 &   1 & 3/2\\
%1/6 & 1/4 &   1\\
%1/6 & 1/4 &   1\\

The best and worst differentiating solution vectors given by $\bm{B}_{2}^{\ast}$ are
\begin{equation*} 
\bm{x}_{2}^{\textup{best}}
=
\left(
\begin{array}{c}
  1 \\ %1.000000000000000
1/4 \\ %0.250000000000000
1/6 \\ %0.166666666666667
1/6    %0.166666666666667
\end{array}
\right),
\qquad
\bm{x}_{2}^{\textup{worst}}
=
\left(
\begin{array}{c}
  1 \\ %1.000000000000000
  1 \\ %1.000000000000000
1/2 \\ %0.500000000000000
1/2    %0.500000000000000
\end{array}
\right),
\end{equation*} 
which shows that these vectors do not coincide.

Because the solution is not unique, we perform the next step and calculate
\begin{equation*} 
\theta_{3}
=
6^{1/3}
\approx
1.8171. %1.817120592832140
\end{equation*}  

Furthermore, we use the above result to obtain the matrices
\begin{equation*} 
\bm{B}_{3}
=
\left(
\begin{array}{cccc}
 1/\theta_{3} &  3/\theta_{3} &          3/2 &            2 \\
          1/4 &  1/\theta_{3} & 2/\theta_{3} & 4/\theta_{3} \\
1/2\theta_{3} & 1/2\theta_{3} & 1/\theta_{3} &            1 \\
1/3\theta_{3} &           1/6 &            1 & 1/\theta_{3}
\end{array}
\right),
\quad
\bm{B}_{3}^{\ast}
=
\left(
\begin{array}{cccc}
            1 & 3/\theta_{3} &  2\theta_{3} &  2\theta_{3} \\
 \theta_{3}/3 &            1 & 4/\theta_{3} & 4/\theta_{3} \\
1/2\theta_{3} & \theta_{3}/4 &            1 &            1 \\
1/2\theta_{3} & \theta_{3}/4 &            1 &            1
\end{array}
\right).
\end{equation*} 

%G4_equiv =
%         1
%6^{1/3}/3
%6^{2/3}/12
%6^{2/3}/12

After evaluation of the best and worst solutions from $\bm{B}_{3}^{\ast}$, we have  
\begin{equation*} 
\bm{x}_{3}^{\textup{best}}
=
\left(
\begin{array}{c}
            1 \\ %1.000000000000000
 \theta_{3}/3 \\ %0.605706864277380
1/2\theta_{3} \\ %0.275160604074552
1/2\theta_{3}    %0.275160604074552
\end{array}
\right),
\qquad 
\bm{x}_{3}^{\textup{worst}}
=
\left(
\begin{array}{c}
            1 \\ %1.000000000000000
 \theta_{3}/3 \\ %0.605706864277380
1/2\theta_{3} \\ %0.275160604074552
1/2\theta_{3}    %0.275160604074552
\end{array}
\right).
\end{equation*} 

Since both solutions coincide, we complete the procedure. As the unique final solution, we take the vector
\begin{equation*} 
\bm{x}
=
\left(
\begin{array}{c}
            1 \\ %1.000000000000000
 \theta_{3}/3 \\ %0.605706864277380
1/2\theta_{3} \\ %0.275160604074552
1/2\theta_{3}    %0.275160604074552
\end{array}
\right)
\approx
\left(
\begin{array}{c}
1.0000 \\ %1.000000000000000
0.6057 \\ %0.605706864277380
0.2752 \\ %0.275160604074552
0.2752    %0.275160604074552
\end{array}
\right).
\end{equation*}

\subsection{Lexicographic Max-Ordering Solution}

At the first step of this solution, we set $I_{0}=\{1,2,3,4\}$ and then obtain
\begin{equation*} 
\bm{A}_{1}
=
\left(
\begin{array}{cccc}
  1 &   3 & 3 & 4 \\
1/2 &   1 & 3 & 4 \\
1/2 &   2 & 1 & 2 \\
  1 & 1/2 & 3 & 1
\end{array}
\right),
\qquad
\theta_{1}
=
3.
\end{equation*} 
 
Furthermore, as in the max-ordering solution, we calculate the matrices
\begin{equation*} 
\bm{B}_{1}
=
\left(
\begin{array}{cccc}
1/3 &   1 &   1 & 4/3 \\
1/6 & 1/3 &   1 & 4/3 \\
1/6 & 2/3 & 1/3 &   1 \\
1/3 & 1/6 &   1 & 1/3
\end{array}
\right),
\qquad 
\bm{B}_{1}^{\ast}
=
\left(
\begin{array}{cccc}
  1 &   1 & 4/3 & 4/3 \\
4/9 &   1 & 4/3 & 4/3 \\
1/3 & 2/3 &   1 &   1 \\
1/3 & 2/3 &   1 &   1
\end{array}
\right),
\end{equation*}  
%G1_equiv =
%  1 &   1 & 4/3\\
%4/9 &   1 & 4/3\\
%1/3 & 2/3 &   1\\
%1/3 & 2/3 &   1\\
and then obtain the best and worst differentiating vectors
\begin{equation*}  
\bm{x}_{1}^{\textup{best}}
=
\left(
\begin{array}{c}
  1 \\ %1.000000000000000 
4/9 \\ %0.444444444444444
1/3 \\ %0.333333333333333
1/3    %0.333333333333333
\end{array}
\right),
\qquad
\bm{x}_{1}^{\textup{worst}}
=
\left(
\begin{array}{c}
  1 \\ %1.000000000000000
  1 \\ %1.000000000000000
3/4 \\ %0.750000000000000
3/4 \\ %0.750000000000000
\end{array}
\right).
\end{equation*} 

We observe that these vectors do not coincide and further calculate
\begin{equation*} 
\theta_{11}
=
3,
\qquad
\theta_{12}
=
2,
\qquad
\theta_{13}
=
8/3,
\qquad
\theta_{14}
=
8/3.
\end{equation*} 
 
Since $\theta_{12},\theta_{13},\theta_{14}<\theta_{1}$, we have the index set $I_{1}=\{2,3,4\}$. 

We turn to the next step and use the set $I_{1}$ to calculate
\begin{equation*} 
\bm{A}_{2}
=
\left(
\begin{array}{cccc}
  1 &   3 & 3 & 4 \\
1/2 &   1 & 2 & 4 \\
1/2 &   2 & 1 & 2 \\
  1 & 1/3 & 1 & 1
\end{array}
\right),
\qquad
\theta_{2}
=
8^{1/2}
\approx
2.8284. %2.828427124746190 
\end{equation*}  

To describe the solution set for this step, we form the matrices
\begin{equation*} 
\bm{B}_{2}
=
\left(
\begin{array}{cccc}
 1/\theta_{2} & 3/\theta_{2} & 3/\theta_{2} & 4/\theta_{2} \\
1/2\theta_{2} & 1/\theta_{2} &            1 & 4/\theta_{2} \\
1/2\theta_{2} & 2/\theta_{2} & 1/\theta_{2} &            1 \\
 1/\theta_{2} &          1/6 &            1 & 1/\theta_{2}
\end{array}
\right),
\quad
\bm{B}_{2}^{\ast}
=
\left(
\begin{array}{cccc}
           1 & 3/\theta_{2} &          3/2 &          3/2 \\
         1/2 &            1 & 4/\theta_{2} & 4/\theta_{2} \\
1/\theta_{2} & 2/\theta_{2} &            1 &            1 \\
1/\theta_{2} & 2/\theta_{2} &            1 &            1
\end{array}
\right)
\end{equation*} 
and then find the vectors
\begin{equation*}  
\bm{x}_{2}^{\textup{best}}
=
\left(
\begin{array}{c}
           1 \\ %1.000000000000000
         1/2 \\ %0.500000000000000
1/\theta_{2} \\ %0.353553390593274
1/\theta_{2}    %0.353553390593274
\end{array}
\right),
\qquad 
\bm{x}_{2}^{\textup{worst}}
=
\left(
\begin{array}{c} 
           1 \\ %1.000000000000000 
\theta_{2}/3 \\ %0.942809041582063
         2/3 \\ %0.666666666666667
         2/3    %0.666666666666667
\end{array}
\right).
\end{equation*} 

Next, we calculate the minimums
\begin{equation*} 
\theta_{22}
=
3\cdot8^{1/2}/4,
\qquad 
\theta_{23}
=
8^{1/2},
\qquad 
\theta_{24}
=
8^{1/2},
\end{equation*} 
which leads to the index set $I_{2}=\{2\}$.
 
According to the set $I_{2}$, we define
\begin{equation*} 
\bm{A}_{3}
=
\left(
\begin{array}{cccc}
1   & 2   & 3   & 4 \\
1/2 & 1   & 2   & 3 \\
1/3 & 1/2 & 1   & 2 \\
1/4 & 1/3 & 1/2 & 1
\end{array}
\right),
\qquad
\theta_{3}
=
3\cdot8^{1/2}/4
\approx
2.1213. %2.121320343559643
\end{equation*} 
 
Furthermore, we calculate the matrices
\begin{equation*} 
\bm{B}_{3}
=
\left(
\begin{array}{cccc}
 1/\theta_{3} & \theta_{3}/2 & 4/\theta_{3} & 4/\theta_{3} \\
 \theta_{3}/9 & 1/\theta_{3} &            1 & 3/\theta_{3} \\
\theta_{3}/12 & \theta_{3}/3 & 1/\theta_{3} &            1 \\
 \theta_{3}/6 &          1/6 &            1 & 1/\theta_{3}
\end{array}
\right),
\quad
\bm{B}_{3}^{\ast}
=
\left(
\begin{array}{cccc}
           1 &          4/3 & \theta_{3}/2 & 4/\theta_{3} \\
         1/2 &            1 & 3/\theta_{3} & 3/\theta_{3} \\
\theta_{3}/6 & \theta_{3}/3 &            1 &            1 \\
\theta_{3}/6 & \theta_{3}/3 &            1 &            1
\end{array}
\right).
\end{equation*}
 
%G3_equiv =
%        1 &       4/3\\
%      1/2 &         1\\
%8^{1/2}/8 & 8^{1/2}/4\\
%8^{1/2}/8 & 8^{1/2}/4\\
 
The best and worst differentiating vectors are given by
\begin{equation*} 
\bm{x}_{3}^{\textup{best}}
=
\left(
\begin{array}{c}
           1 \\ %1.000000000000000
         1/2 \\ %0.500000000000000
\theta_{3}/6 \\ %0.353553390593274
\theta_{3}/6    %0.353553390593274
\end{array}
\right)
\approx
\left(
\begin{array}{c}
1.0000 \\ %1.000000000000000
0.5000 \\ %0.500000000000000
0.3536 \\ %0.353553390593274
0.3536    %0.353553390593274
\end{array}
\right),
\qquad 
\bm{x}_{3}^{\textup{worst}}
=
\left(
\begin{array}{c}
           1 \\ %1.000000000000000
         3/4 \\ %0.750000000000000
\theta_{3}/4 \\ %0.530330085889911
\theta_{3}/4    %0.530330085889911
\end{array}
\right)
\approx
\left(
\begin{array}{c}
1.0000 \\ %1.000000000000000
0.7500 \\ %0.750000000000000
0.5303 \\ %0.530330085889911
0.5303    %0.530330085889911
\end{array}
\right).
\end{equation*}  

Since at this step we have $I_{3}=\emptyset$, the procedure terminates. We can couple both obtained vectors into one vector with interval entries as
\begin{equation*} 
\bm{x}
\approx
\left(
\begin{array}{c}
                1.0000 \\ %1.000000000000000
0.5000\ \ldots\ 0.7500 \\ %0.500000000000000-0.750000000000000
0.3536\ \ldots\ 0.5303 \\ %0.353553390593274-0.530330085889911
0.3536\ \ldots\ 0.5303    %0.353553390593274-0.530330085889911
\end{array}
\right).
\end{equation*}  

To conclude we observe that all solutions result in the same preference order of alternatives in the form $(1)\succ(2)\succ(3)\equiv(4)$, whereas the individual ratings of alternatives may differ. We note that this order is in agreement with the above conclusions based on straightforward analysis of pairwise comparison matrices, and satisfies the constraint imposed on the ratings.

\section{Solution to Vacation Plan Problem}
\label{S-SVPP}

As another example which offers a potential to consider the proposed approach in line with existing methods, we solve the problem of selecting a plan for vacation from \cite{Saaty1977Scaling}. The aim is to choose a destination for vacation trip from Philadelphia among the alternatives $\mathbf{S}$: short trips (i.e., New York, Washington, Atlantic City, New Hope, etc.), $\mathbf{Q}$: Quebec, $\mathbf{D}$: Denver, $\mathbf{C}$: California. These places are evaluated in terms of the following criteria: (1) cost of the trip from Philadelphia, (2) sight-seeing opportunities, (3) entertainment (doing things), (4) way of travel, (5) eating places. 

The results of pairwise comparison of criteria are given by the matrix
\begin{equation*}
\bm{C}_{0}
=
\left(
\begin{array}{ccccc}
1 & 1/5 & 1/5 & 1 & 1/3 \\
5 & 1 & 1/5 & 1/5 & 1 \\
5 & 5 & 1 & 1/5 & 1 \\
1 & 5 & 5 & 1 & 5 \\
3 & 1 & 1 & 1/5 & 1
\end{array}
\right).
\end{equation*}

The pairwise comparison matrices of places with respect to the criteria are as follows:
\begin{gather*}
\bm{C}_{1}
=
\left(
\begin{array}{cccc}
  1 &   3 &   7 & 9 \\
1/3 &   1 &   6 & 7 \\
1/7 & 1/6 &   1 & 3 \\
1/9 & 1/7 & 1/3 & 1
\end{array}
\right),
\qquad
\bm{C}_{2}
=
\left(
\begin{array}{cccc}
1 & 1/5 & 1/6 & 1/4 \\
5 &   1 &   2 & 4 \\
6 & 1/2 &   1 & 6 \\
4 & 1/4 & 1/6 & 1
\end{array}
\right),
\\
\bm{C}_{3}
=
\left(
\begin{array}{cccc}
  1 & 7 & 7 & 1/2 \\
1/7 & 1 & 1 & 1/7 \\
1/7 & 1 & 1 & 1/7 \\
  2 & 7 & 7 & 1
\end{array}
\right),
\qquad
\bm{C}_{4}
=
\left(
\begin{array}{cccc}
  1 &   4 & 1/4 & 1/3 \\
1/4 &   1 & 1/2 & 3 \\
  4 &   2 &   1 & 3 \\
  3 & 1/3 & 1/3 & 1
\end{array}
\right),
\\
\bm{C}_{5}
=
\left(
\begin{array}{cccc}
  1 &   1 & 7 & 4 \\
  1 &   1 & 6 & 3 \\
1/7 & 1/6 & 1 & 1/4 \\
1/4 & 1/3 & 4 & 1
\end{array}
\right).
\end{gather*}

To solve the problem, the analytical hierarchy process method is applied in \cite{Saaty1977Scaling}, which provides the following order of alternatives: $\mathbf{S}\succ\mathbf{D}\succ\mathbf{C}\succ\mathbf{Q}$. Another solution based on the log-Chebyshev approximation in \cite{Krivulin2019Tropical} yields a different order $\mathbf{C}\succeq\mathbf{S}\succ\mathbf{D}\succeq\mathbf{Q}$.
Below, we show how the new approach can be used to obtain solutions according to various principles of optimality.

\subsection{Max-Ordering Solution}

We start with the max-ordering solution given by Theorem~\ref{T-minx_xC1x_xCmx-Bxleqx_MO}. Observing that $\bm{B}_{0}=\bm{B}=\bm{0}$, we calculate
\begin{equation*}
\bm{A}
=
\left(
\begin{array}{cccc}
1 & 7 & 7 & 9 \\
5 & 1 & 6 & 7 \\
6 & 2 & 1 & 6 \\
4 & 7 & 7 & 1
\end{array}
\right),
\qquad
\theta
=
3\cdot14^{1/3}
\approx
7.2304, %7.230426792525689
\end{equation*}

Next, we form the matrices
\begin{equation*} 
\bm{B}_{1}
=
\theta^{-1}
\left(
\begin{array}{cccc}
1 & 7 & 7 & 9 \\
5 & 1 & 6 & 7 \\
6 & 2 & 1 & 6 \\
4 & 7 & 7 & 1
\end{array}
\right),
\qquad
\bm{G}
=
\bm{B}_{1}^{\ast}
=
\left(
\begin{array}{cccc}
       1 & \theta/6 &   \theta/6 & 9/\theta \\
     7/9 &        1 & 7\theta/54 & 7/\theta \\
6/\theta &        1 &          1 & \theta/7 \\
\theta/9 & 7/\theta &   7/\theta & 1
\end{array}
\right).
\end{equation*}

%G_equiv =
%[  1,   \theta/6]
%[7/9,   1]
%[6/\theta, 1]
%[\theta/9, 7/theta]
The minimal normalized best differentiating and maximal normalized worst differentiating solutions obtained from $\bm{G}$ takes the form
\begin{equation*} 
\bm{x}^{\textup{best}}
=
\left(
\begin{array}{c}
1 \\
7/9 \\
6/\theta \\
\theta/9
\end{array}
\right)
\approx
\left(
\begin{array}{c}
1.0000 \\ %1.000000000000000
0.7778 \\ %0.777777777777778
0.8298 \\ %0.829826533366243
0.8034 \\ %0.803380754725077
\end{array}
\right),
\qquad 
\bm{x}^{\textup{worst}}
=
\left(
\begin{array}{c}
1 \\
6/\theta \\
6/\theta \\
\theta/9
\end{array}
\right)
\approx
\left(
\begin{array}{c}
1.0000 \\ %1.000000000000000
0.8298 \\ %0.829826533366243
0.8298 \\ %0.829826533366243
0.8034 %0.803380754725077
\end{array}
\right),
\end{equation*} 
and specify the respective orders $\mathbf{S}\succ\mathbf{D}\succ\mathbf{C}\succ\mathbf{Q}$ and $\mathbf{S}\succ\mathbf{D}\equiv\mathbf{Q}\succ\mathbf{C}$.

Both solutions can be written as the vector
\begin{equation*} 
\bm{x}
\approx
\left(
\begin{array}{c}
                1.0000 \\ %1.000000000000000-1.000000000000000
0.7778\ \ldots\ 0.8298 \\ %0.777777777777778-0.829826533366243
                0.8298 \\ %0.829826533366243-0.829826533366243
                0.8034 %0.803380754725077-0.803380754725077
\end{array}
\right).
\end{equation*} 

The obtained solutions result in the order $\mathbf{S}\succ\mathbf{D}\succeq\mathbf{C}\parallel\mathbf{Q}$, where the alternatives $\mathbf{C}$ and $\mathbf{Q}$ cannot be uniquely ordered.

\subsection{Lexicographic Ordering Solution}

In order to obtain a lexicographic ordering solution, we first need to rank the criteria from the pairwise comparison matrix $\bm{C}_{0}$. By comparing the rows of the matrix $\bm{C}_{0}$, one can expect that the most important is criterion $4$ and the second is $3$. Next come criteria $2$ and $5$, and finally $1$ as the least important criterion. These considerations are confirmed by the result in \cite{Saaty1977Scaling}, where a priority vector is calculated from $\bm{C}_{0}$ by using the principal eigenvector method, which puts the criteria into the order $4,3,2,5,1$.

To simplify the application of the lexicographic ordering based solutions in what follows, we renumber the criteria in the problem according to their importance. Specifically, we assign the number $1$ one to criterion $4$, the number $2$ to criterion $3$ and so on. We use the new numbers of criteria to rename the pairwise comparison matrices for alternatives as follows:
\begin{gather*}
\bm{C}_{1}
=
\left(
\begin{array}{cccc}
  1 &   4 & 1/4 & 1/3 \\
1/4 &   1 & 1/2 & 3 \\
  4 &   2 &   1 & 3 \\
  3 & 1/3 & 1/3 & 1
\end{array}
\right),
\qquad
\bm{C}_{2}
=
\left(
\begin{array}{cccc}
  1 & 7 & 7 & 1/2 \\
1/7 & 1 & 1 & 1/7 \\
1/7 & 1 & 1 & 1/7 \\
  2 & 7 & 7 & 1
\end{array}
\right),
\\
\bm{C}_{3}
=
\left(
\begin{array}{cccc}
1 & 1/5 & 1/6 & 1/4 \\
5 &   1 &   2 & 4 \\
6 & 1/2 &   1 & 6 \\
4 & 1/4 & 1/6 & 1
\end{array}
\right),
\qquad
\bm{C}_{4}
=
\left(
\begin{array}{cccc}
  1 &   1 & 7 & 4 \\
  1 &   1 & 6 & 3 \\
1/7 & 1/6 & 1 & 1/4 \\
1/4 & 1/3 & 4 & 1
\end{array}
\right),
\\
\bm{C}_{5}
=
\left(
\begin{array}{cccc}
  1 &   3 &   7 & 9 \\
1/3 &   1 &   6 & 7 \\
1/7 & 1/6 &   1 & 3 \\
1/9 & 1/7 & 1/3 & 1
\end{array}
\right).
\end{gather*}

We now solve the problem by applying Theorem~\ref{T-minx_xC1x_xCmx-Bxleqx_LO}. We first calculate  
\begin{equation*}
\theta_{1}
=
36^{1/3}
\approx
3.3019, %3.301927248894627
\end{equation*}
and then construct the matrices
\begin{equation*}
\bm{B}_{1}
=
\theta_{1}^{-1}
\left(
\begin{array}{cccc}
  1 &   4 & 1/4 & 1/3 \\
1/4 &   1 & 1/2 & 3 \\
  4 &   2 &   1 & 3 \\
  3 & 1/3 & 1/3 & 1
\end{array}
\right),
\qquad
\bm{B}_{1}^{\ast}
=
\left(
\begin{array}{cccc}
           1 &  4/\theta_{1} & \theta_{1}/18 & \theta_{1}/3 \\
\theta_{1}/4 &             1 & 1/2\theta_{1} & 3/\theta_{1} \\
4/\theta_{1} & 4\theta_{1}/9 &             1 & 4/3 \\
3/\theta_{1} &  \theta_{1}/3 &           1/6 & 1
\end{array}
\right).
\end{equation*}

%G1_equiv =
%1 & \theta_{1}/18\\
%\theta_{1}/4 & 1/2\theta_{1}\\
%4/\theta_{1} & 1\\
%3/\theta_{1} & 1/6\\

Calculation of the best and worst differentiating solutions from $\bm{B}_{1}^{\ast}$ yields
\begin{equation*}
\bm{x}_{1}^{\textup{best}}
=
\left(
\begin{array}{c}
\theta_{1}/18 \\ %0.183440402716368 
1/2\theta_{1} \\ %0.151426716069345
1 \\ %1.000000000000000
1/6 %0.166666666666667
\end{array}
\right),
\qquad 
\bm{x}_{1}^{\textup{worst}}
=
\left(
\begin{array}{c}
\theta_{1}/4 \\ %0.825481812223657
9/4\theta_{1} \\ %0.681420222312052
1 \\ %1.000000000000000
3/4 %0.750000000000000
\end{array}
\right).
\end{equation*} 

Observing that the obtained vectors do not coincide, we further evaluate
\begin{equation*}
\theta_{2}
=
28/3
\approx
9.3333, %9.333333333333334
\end{equation*} 
and find the matrices
\begin{equation*}
\bm{B}_{2}
=
\left(
\begin{array}{cccc}
 1/\theta_{1} & 4/\theta_{1} &           3/4 & 1/3\theta_{1} \\
1/4\theta_{1} & 1/\theta_{1} & 1/2\theta_{1} & 3/\theta_{1} \\
 4/\theta_{1} & 2/\theta_{1} &  1/\theta_{1} & 3/\theta_{1} \\
 3/\theta_{1} &          3/4 &           3/4 & 1/\theta_{1}
\end{array}
\right),
\qquad
\bm{B}_{2}^{\ast}
=
\left(
\begin{array}{cccc}
           1 &  4/\theta_{1} &  \theta_{1}/4 & \theta_{1}/3 \\
\theta_{1}/4 &             1 & 9/4\theta_{1} & 3/\theta_{1} \\
4/\theta_{1} & 4\theta_{1}/9 &             1 & 4/3 \\
3/\theta_{1} &  \theta_{1}/3 &           3/4 & 1
\end{array}
\right).
\end{equation*}
 
%G2_equiv =
%\theta_{1}/4 \\
%9/4\theta_{1} \\
%1 \\
%3/4

Further calculation shows that the best and worst differentiating solutions given by $\bm{B}_{2}^{\ast}$ coincide and thus provide a unique solution to the problem as
\begin{equation*} 
\bm{x}
=
\bm{x}_{2}^{\textup{best}}
=
\bm{x}_{2}^{\textup{worst}}
=
\left(
\begin{array}{c}
\theta_{1}/4 \\ %0.825481812223657
9/4\theta_{1} \\ %0.681420222312052
1 \\ %1.000000000000000
3/4 \\ %0.750000000000000
\end{array}
\right)
\approx
\left(
\begin{array}{c}
0.8255 \\ %0.825481812223657
0.6814 \\ %0.681420222312052
1.0000 \\ %1.000000000000000
0.7500 %0.750000000000000
\end{array}
\right).
\end{equation*} 

The obtained solution gives the order $\mathbf{D}\succ\mathbf{S}\succ\mathbf{C}\succ\mathbf{Q}$.

\subsection{Lexicographic Max-Ordering Solution}

We now apply Theorem~\ref{T-minx_xC1x_xCmx-Bxleqx_LMO} to obtain the lexicographic max-ordering solution. We set $I_{0}=\{1,2,3,4\}$ and calculate
\begin{equation*}
\bm{A}_{1}
=
\left(
\begin{array}{cccc}
1 & 7 & 7 & 9 \\
5 & 1 & 6 & 7 \\
6 & 2 & 1 & 6 \\
4 & 7 & 7 & 1
\end{array}
\right),
\qquad
\theta_{1}
=
3\cdot14^{1/3}
\approx
7.2304. %7.230426792525689
\end{equation*}

Furthermore, we obtain the matrices
\begin{equation*} 
\bm{B}_{1}
=
\theta^{-1}
\left(
\begin{array}{cccc}
1 & 7 & 7 & 9 \\
5 & 1 & 6 & 7 \\
6 & 2 & 1 & 6 \\
4 & 7 & 7 & 1
\end{array}
\right),
\qquad
\bm{B}_{1}^{\ast}
=
\left(
\begin{array}{cccc}
       1 & \theta/6 &   \theta/6 & 9/\theta \\
     7/9 &        1 & 7\theta/54 & 7/\theta \\
6/\theta &        1 &          1 & \theta/7 \\
\theta/9 & 7/\theta &   7/\theta & 1
\end{array}
\right).
\end{equation*}

%G_equiv =
%[  1,   \theta/6]
%[7/9,   1]
%[6/\theta, 1]
%[\theta/9, 7/theta]
The best and worst differentiating solutions are given by the vectors
\begin{equation*} 
\bm{x}^{\textup{best}}
=
\left(
\begin{array}{c}
1 \\
7/9 \\
6/\theta \\
\theta/9
\end{array}
\right)
\approx
\left(
\begin{array}{c}
1.0000 \\ %1.000000000000000
0.7778 \\ %0.777777777777778
0.8298 \\ %0.829826533366243
0.8034 \\ %0.803380754725077
\end{array}
\right),
\qquad 
\bm{x}^{\textup{worst}}
=
\left(
\begin{array}{c}
1 \\
6/\theta \\
6/\theta \\
\theta/9
\end{array}
\right)
\approx
\left(
\begin{array}{c}
1.0000 \\ %1.000000000000000
0.8298 \\ %0.829826533366243
0.8298 \\ %0.829826533366243
0.8034 %0.803380754725077
\end{array}
\right).
\end{equation*} 

Considering that these vectors do not coincide, we further calculate
\begin{equation*} 
\theta_{11}
=
2\theta_{1}/3,
\qquad
\theta_{12}
=
\theta_{1},
\qquad
\theta_{13}
=
\theta_{1},
\qquad
\theta_{14}
=
6,
\qquad
\theta_{15}
=
\theta_{1}.
\end{equation*} 

We observe that $\theta_{11},\theta_{14}<\theta_{1}$ and take the set $I_{1}=\{1,4\}$ to calculate
\begin{equation*} 
\bm{A}_{2}
=
\left(
\begin{array}{cccc}
1 &   4 & 7 & 4 \\
1 &   1 & 6 & 3 \\
4 &   2 & 1 & 3 \\
3 & 1/3 & 4 & 1
\end{array}
\right),
\qquad
\theta_{2}
=
6.
\end{equation*}  

Further calculations yield the matrices
\begin{equation*} 
\bm{B}_{2}
=
\left(
\begin{array}{cccc}
         1/6 & 7/\theta_{1} & 7/6          & 9/\theta_{1} \\
5/\theta_{1} &          1/6 &            1 & 7/\theta_{1} \\
6/\theta_{1} &          1/3 &          1/6 & 6/\theta_{1} \\
4/\theta_{1} & 7/\theta_{1} & 7/\theta_{1} & 1/6
\end{array}
\right),
\quad
\bm{B}_{2}^{\ast}
=
\left(
\begin{array}{cccc}
           1 & \theta_{1}/6 & \theta_{1}/6 & 9/\theta_{1} \\
6/\theta_{1} &            1 &            1 & \theta_{1}/7 \\
6/\theta_{1} &            1 &            1 & \theta_{1}/7 \\
\theta_{1}/9 & 7/\theta_{1} & 7/\theta_{1} &            1
\end{array}
\right).
\end{equation*} 

Evaluation of the best and worst differentiating solutions yields the vector
\begin{equation*}  
\bm{x}
=
\bm{x}_{2}^{\textup{best}}
=
\bm{x}_{2}^{\textup{worst}}
=
\left(
\begin{array}{c}
           1 \\ %1.000000000000000
6/\theta_{1} \\ %0.829826533366243
6/\theta_{1} \\ %0.829826533366243
\theta_{1}/9 %0.803380754725077
\end{array}
\right)
\approx
\left(
\begin{array}{c}
1.0000 \\ %1.000000000000000
0.8298 \\ %0.829826533366243
0.8298 \\ %0.829826533366243
0.8034 %0.803380754725077
\end{array}
\right),
\end{equation*} 
which places the alternatives in the order $\mathbf{S}\succ\mathbf{D}\equiv\mathbf{Q}\succ\mathbf{C}$.

A comparison of the results obtained shows that the solutions obtained according to the max-ordering and lexicographic max-ordering principles of optimality and the solution obtained by the analytical hierarchy process method in \cite{Saaty1977Scaling} rank the alternative $\mathbf{S}$ higher than the others. A different result provided by lexicographic ordering is explained by the fact that this technique can overestimate the contribution of pairwise comparisons made according to the most important criterion, and thereby provide little or no consideration of the other comparisons. Finally, we note that the order obtained by the analytical hierarchy process method and the order given by the best differentiating max-ordering solution completely coincide.

\newpage
\section{Conclusion}
\label{S-C}

In this paper, we have considered a decision-making problem of rating alternatives, based on pairwise comparisons under multiple criteria, subject to constraints imposed on ratings. The problem was reduced to constrained log-Chebyshev approximation of pairwise comparison matrices by a common consistent matrix that directly determines the vector of ratings. We have represented the approximation problem as a constrained multiobjective optimization problem in terms of tropical algebra where the addition is defined as maximum. We have applied methods and results of tropical optimization to solve the multiobjective problem according to the max-ordering, lexicographic ordering and lexicographic max-ordering principles of optimality. The proposed technique was illustrated with numerical solutions of some pairwise comparison problems, including the problem of selecting a plan for vacation from \cite{Saaty1977Scaling}.

The obtained results show that the application of the log-Chebyshev approximation in combination with tropical algebra allows one to obtain analytical solutions ready for both formal analysis and straightforward computations. Although the new technique to solve pairwise comparison problems in general involves more computational effort than the existing principal eigenvector and geometric mean methods, the computational complexity of this technique remains of a moderate degree and does not exceed $O(m^{2}n^{5})$ where $m$ is the number of criteria and $n$ is the number of alternatives. At the same time, it offers an advantage over the above two methods that cannot handle constraints in a direct and efficient way, and thus do not provide easy implementation of the optimality principles which involve solving constrained optimization problems.

The application of the log-Chebyshev approximation can yield multiple solutions to the pairwise comparison problem. Although the multiple solutions seem to correspond well to the approximate nature of the model of pairwise comparisons, which in most cases is poorly consistent with the data, the existence of many optimal solutions can make it difficult to find an appropriate decision. To overcome this possible drawback, we have developed a new technique that reduces the multiple solutions to the ``best'' and ``worst'' solutions that can be reasonably representative of the entire solution set and hence simplify the identification and selection of the right decision.

The technique examines the solution vectors of ratings, which are normalized to have the maximum value of $1$. As the best solution, a minimal vector is taken that maximizes the ratio between the maximum and minimum components (the highest and lowest ratings). The worst solution is the maximal vector that minimizes the ratio between the maximum and minimum components. While the worst solution is always unique, there may be more than one best solution, which can be considered as a limitation of the technique. Note however that multiple best solutions (the number of which cannot exceed $n$ either) should be explained as a result of unstable pairwise comparison data and hence indicate unreliable input information, which is difficult to interpret unambiguously. 

The main contribution of the study is twofold. First, we have formulated and investigated new constrained multicriteria decision problems that model complex decision-making processes, but received no attention in the literature. We have proposed an approach based on the log-Chebyshev approximation, which in contrast to other methods in decision making makes it possible to obtain analytical solutions to the problems of interest. Moreover, new formal methods and real-world applications developed to extend the mathematical instruments and broaden the application scope of tropical optimization, present another significant outcome. 

Practical implications of the study are that it suggests new practically applicable solution techniques to supplement and complement existing methods in the situation when these methods are known to produce different and even opposite results. In the case of conflicting results, the new technique can serve to obtain additional arguments in favor or against solutions offered by other methods, and thereby help make more accurate decisions. As another practical consequence, one can consider the ability to solve real-world multicriteria decision problems with constraints, where weights of criteria are not specified or the criteria are differentiated only by ranks. 

Possible lines of further investigation include extension of the results to solve multicriteria problems under different types of constraints and principles of optimality. The development of interval and fuzzy implementations of the proposed solutions is another area of interest for future research.

\section*{Acknowledgments}
The author is very grateful to two anonymous reviewers for their valuable comments and suggestions, which have been incorporated in the revised manuscript.

\bibliographystyle{abbrvurl}

\bibliography{Application_of_tropical_optimization_for_solving_multicriteria_problems_of_pairwise_comparisons_using_log-Chebyshev_approximation}

\end{document}